\pdfoutput=1
\documentclass[11pt,a4paper,reqno]{amsart}

\usepackage[british]{babel}

\usepackage[DIV=9,oneside,BCOR=0mm]{typearea}
\usepackage{microtype}
\usepackage[T1]{fontenc}
\usepackage{setspace}
\usepackage{amssymb}
\usepackage{amsmath}
\usepackage{amsthm}
\usepackage{enumitem}
\usepackage{mathrsfs}
\usepackage[colorlinks,citecolor=blue,urlcolor=black,linkcolor=blue,linktocpage]{hyperref}
\usepackage{mathtools}
\usepackage{xcolor}
\usepackage{xfrac}

\usepackage{accents}

\newtheorem{innercustomgeneric}{\customgenericname}
\providecommand{\customgenericname}{}
\newcommand{\newcustomtheorem}[2]{%
	\newenvironment{#1}[1]
	{%
		\renewcommand\customgenericname{#2}%
		\renewcommand\theinnercustomgeneric{##1}%
		\innercustomgeneric
	}
	{\endinnercustomgeneric}
}

\newcustomtheorem{customthm}{Theorem}
\newcustomtheorem{customcor}{Corollary}

\newtheorem{thm}{Theorem}[section]

\newtheorem{cor}[thm]{Corollary}
\newtheorem{lem}[thm]{Lemma}
\newtheorem{prop}[thm]{Proposition}
\newtheorem{problem}[thm]{Problem}

\theoremstyle{definition}

\newtheorem{definition}[thm]{Definition}

\newtheorem{remark}[thm]{Remark}

\renewcommand{\epsilon}{\varepsilon}
\renewcommand{\phi}{\varphi}
\newcommand{\defeq}{\mathrel{\mathop:}=}                         
\newcommand{\eqdef}{\mathrel{\mathopen={\mathclose:}}}

\DeclareMathOperator{\Sym}{Sym}
\DeclareMathOperator{\Alt}{Alt}
\DeclareMathOperator{\free}{F}
\DeclareMathOperator{\PSL}{PSL}
\DeclareMathOperator{\ct}{A}
\DeclareMathOperator{\sr}{sr}
\DeclareMathOperator{\cha}{char}

\DeclareMathOperator{\EL}{EL}
\DeclareMathOperator{\rk}{rk}
\DeclareMathOperator{\rank}{rank}
\DeclareMathOperator{\M}{M}
\DeclareMathOperator{\I}{I}
\DeclareMathOperator{\lat}{L}
\DeclareMathOperator{\E}{E}
\DeclareMathOperator{\GL}{GL}
\DeclareMathOperator{\PGL}{PGL}
\DeclareMathOperator{\SL}{SL}
\DeclareMathOperator{\cent}{Z}
\DeclareMathOperator{\Cl}{Cl}

\DeclareMathOperator{\N}{\mathbb{N}}

\makeatletter
\def\moverlay{\mathpalette\mov@rlay}
\def\mov@rlay#1#2{\leavevmode\vtop{%
		\baselineskip\z@skip \lineskiplimit-\maxdimen
		\ialign{\hfil$\m@th#1##$\hfil\cr#2\crcr}}}
\newcommand{\charfusion}[3][\mathord]{
	#1{\ifx#1\mathop\vphantom{#2}\fi
		\mathpalette\mov@rlay{#2\cr#3}
	}
	\ifx#1\mathop\expandafter\displaylimits\fi}
\makeatother

\makeatletter
\newcommand\niton{\mathrel{\m@th\mathpalette\canc@l\owns}}
\newcommand\canc@l[2]{{\ooalign{$\hfil#1/\mkern1mu\hfil$\crcr$#1#2$}}}
\makeatother

\begin{document}

\setlist{noitemsep}

\author{Friedrich Martin Schneider}
\address{F.M.~Schneider, Institute of Discrete Mathematics and Algebra, TU Bergakademie Freiberg, 09596 Freiberg, Germany}
\email{martin.schneider@math.tu-freiberg.de}
\thanks{This research is funded by the Deutsche Forschungsgemeinschaft (DFG, German Research Foundation) -- Projektnummer 561178190}

\title[Projective simplicity and Bergman's property for unit groups]{Projective simplicity and Bergman's property for unit groups of continuous rings}
\date{\today}

\begin{abstract} We prove that the projective unit group $\PGL(R)$, i.e., the quotient of the unit group $\GL(R)$ modulo its center, of any non-discrete irreducible, continuous ring $R$ is simple. Moreover, we show that $\GL(R)$ has uncountable strong cofinality, that is, it is not the union of a countable chain of proper subgroups and it has finite width with respect to any generating set. Equivalently, every isometric action of $\GL(R)$ on a metric space has bounded orbits. It follows that every action of $\GL(R)$ by isometries on a non-empty complete $\mathrm{CAT}(0)$ space admits a fixed point. In particular, $\GL(R)$ possesses Serre's properties $(FH)$ and $(FA)$. Furthermore, our results entail that $\PGL(R)$ has bounded normal generation. In turn, we answer two questions by Carderi and Thom. Finally, building on work of Larsen and Shalev, we prove that $\GL(R)$ is verbally simple.
\end{abstract}



\maketitle

\allowdisplaybreaks


\tableofcontents

\newpage

\section{Introduction}

A classical theorem collectively due to Jordan~\cite{jordan}, Moore~\cite{moore}, and Dickson~\cite{dickson1,dickson2} asserts that for any field $F$ and any integer $n \geq 2$ such that $(n,\vert F \vert) \notin \{ (2,2), (2,3) \}$ the projective special linear group \begin{displaymath}
	\PSL_{n}(F) \, = \, \SL_{n}(F)/\cent(\SL_{n}(F))
\end{displaymath} is simple, i.e., it does not a proper non-trivial normal subgroup. For a field $F$, Carderi and Thom~\cite{CarderiThom} considered the inductive limit $\ct_{0}(F)$ of the special linear groups \begin{displaymath}
	\SL_{2^{0}}(F) \, \lhook\joinrel\longrightarrow \, \ldots \, \lhook\joinrel\longrightarrow \, \SL_{2^{n}}(F) \, \lhook\joinrel\longrightarrow \, \SL_{2^{n+1}}(F) \, \lhook\joinrel\longrightarrow \, \ldots
\end{displaymath} with respect to the diagonal embeddings \begin{displaymath}
	\SL_{2^{n}}(F) \, \lhook\joinrel\longrightarrow \,  \SL_{2^{n+1}}(F), \quad a \, \longmapsto \, 
	\begin{pmatrix}
		a & 0\\
		0 & a
	\end{pmatrix} \qquad (n \in \N) .
\end{displaymath} They observed that those embeddings are isometric relative to the bi-invariant metrics \begin{displaymath}
	\SL_{2^{n}}(F) \times \SL_{2^{n}}(F) \, \longrightarrow \, [0,1], \quad (a,b) \, \longmapsto \, \tfrac{\rank(a-b)}{2^{n}} \qquad (n \in \N) ,
\end{displaymath} wherefore the latter jointly extend to a bi-invariant metric on $\ct_{0}(F)$ and the corresponding metric completion $\ct(F)$ of $\ct_{0}(F)$ naturally constitutes a group. Building on the work of Liebeck and Shalev~\cite{LiebeckShalev}, Carderi and Thom~\cite{CarderiThom} showed that, for a finite field $F$, the center~$\cent(\ct(F))$ is isomorphic to the multiplicative group~$F\setminus \{ 0 \}$ and $\ct(F)/\cent(\ct(F))$ is topologically simple with respect to the topology induced by the metric described above. What is more, they asked the following questions.

\begin{problem}[{\cite[p.~260]{CarderiThom}}\footnote{In~\cite{CarderiThom}, this was asked under the additional hypothesis of $F$ being finite.}]\label{problem:carderi.thom} Let $F$ be a field. Is the group $\ct(F)/\cent(\ct(F))$ simple? If so, does $\ct(F)/\cent(\ct(F))$ have bounded normal generation? \end{problem}

A group $G$ is said to have \emph{bounded normal generation} if, for every $g \in G\setminus \{ 1\}$, there exists $n \in \N$ such that every element of $G$ is a product of~$n$ elements of the union of the conjugacy classes of $g$ and $g^{-1}$. Evidently, bounded normal generation implies simplicity.

The present manuscript answers both of the two questions stated in Problem~\ref{problem:carderi.thom} in the affirmative (see Theorem~\ref{theorem:simple} and Corollary~\ref{corollary:bng} below).

As noticed by Carderi and Thom~\cite{CarderiThom}, any group of the kind above may be viewed as the \emph{unit group} $\GL(R)$, i.e., the group of multiplicatively invertible elements, of the coordinate ring $R$ of some \emph{continuous geometry}~\cite{VonNeumannBook}. More precisely, a \emph{continuous ring} is a von Neumann regular ring $R$ whose lattice of principal right ideals constitutes a continuous geometry. By work of von Neumann~\cite{VonNeumannBook}, every (directly) irreducible, continuous ring $R$ admits a unique \emph{rank function}, which in turn gives rise to a distinguished metric on $R$ and thus furnishes the ring $R$ with a compatible topology---the \emph{rank topology} (see Section~\ref{section:preliminaries} for details). The rank topology of an irreducible, continuous ring is discrete if and only if the ring is isomorphic to a matrix ring over some division ring (see, e.g.,~\cite[Remark~3.6]{SchneiderIMRN}). Instances of irreducible, continuous rings with non-discrete rank topology have emerged in the study of operator algebras~\cite{MurrayVonNeumann,LinnellSchick,elek} and group rings~\cite{ElekSzabo,linnell,elek2013}. An intriguingly elementary example of such a ring arises from any field $F$: the metric completion $\M_{\infty}(F)$ of the inductive limit of the rings $\M_{2^{n}}(F)$ $(n \in \N)$ along the embeddings \begin{displaymath}
	\M_{2^{n}}(F) \, \lhook\joinrel\longrightarrow \, \M_{2^{n+1}}(F), \quad a \, \longmapsto \, \begin{pmatrix} a & 0 \\ 0 & a \end{pmatrix} \qquad (n \in \N) 
\end{displaymath} with respect to the unique metric jointly extending the normalized rank distances \begin{displaymath}
	\M_{2^{n}}(F) \times \M_{2^{n}}(F) \, \longrightarrow \, [0,1], \quad (a,b) \, \longmapsto \, \tfrac{\rank(a-b)}{2^{n}} \qquad (n \in \N)
\end{displaymath} is a non-discrete irreducible, continuous ring~\cite{NeumannExamples,Halperin68}. Its unit group~$\GL(\M_{\infty}(F))$ is isomorphic to $\ct(F)$~\cite[p.~258]{CarderiThom} (see also~\cite[Lemma~6.14]{Schneider25}).

The purpose of the paper at hand is the study of the \emph{projective unit group} $\PGL(R)$ of an arbitrary irreducible, continuous ring $R$, that is, the quotient of $\GL(R)$ modulo its center $\cent(\GL(R))$. It should be noted that $\cent(\GL(R))$ coincides with the set of non-zero elements in the ring's center $\cent(R)$~\cite[Proposition~3.11]{Schneider25}. In particular, for any field $F$, since $\cent(\M_{\infty}(F)) \cong F$ by a result of Goodearl~\cite[Theorem~2.8(c)]{Goodearl78}, it follows that $\cent(\ct(F)) \cong \cent(\GL(\M_{\infty}(F))) \cong F \setminus \{ 0 \}$. Building on, among other things, the work of Bass~\cite{bass} and some recent insights from~\cite{BernardSchneider}, we prove the following.

\begin{thm}\label{theorem:simple} Let $R$ be a non-discrete irreducible, continuous ring. Then $\PGL(R)$ is a simple group. \end{thm}

One route from simplicity to bounded normal generation proceeds via an algebraic phenomenon of boundedness, which was originally discovered by Bergman in~\cite{bergman} and which has since been studied extensively (see, e.g.,~\cite{DrosteGoebel,DrosteHolland,TolstykhLinear,TolstykhFree,Cornulier06,DrosteHollandUlbrich,KechrisRosendal,RicardRosendal,rosendal,dowerk}). A group $G$ is said to have \emph{uncountable strong cofinality}~\cite{DrosteGoebel} if, for every ascending chain $(W_{n})_{n \in \N}$ of subsets of $G$ with $\bigcup_{n \in \N} W_{n} = G$, there exist $m,n \in \N$ such that $W_{n}^{m} = G$. Equivalently, a group possesses this property if and only if it has \begin{itemize}
	\item[---\,] \emph{uncountable cofinality}, i.e., it is not the union of a countable chain of proper subgroups, and
	\item[---\,] the \emph{Bergman property}, that is, it has finite width with respect to every generating set (see Section~\ref{section:bergman.abstract} for details and references).
\end{itemize} One readily observes that simple groups with the Bergman property have bounded normal generation (Lemma~\ref{lemma:bergman.simple}).

We develop an abstract criterion for uncountable strong cofinality of unit groups of unital rings (Theorem~\ref{theorem:abstract}), which we then combine with results of~\cite{Ehrlich56} and~\cite{BernardSchneider} to prove the following.

\begin{thm}\label{theorem:bergman} Let $R$ be a non-discrete irreducible, continuous ring. Then $\GL(R)$ has uncountable strong cofinality. \end{thm}	

The non-discreteness assumption in Theorem~\ref{theorem:bergman} is essential. Indeed, for every infinite field $F$ and every positive integer $n$, the unit group $\GL(\M_{n}(F)) = \GL_{n}(F)$ does not have the Bergman property~\cite[Theorem~1(ii)]{jschneider}, nor uncountable cofinality (for uncountable $F$, this is established in~\cite[p.~318]{jschneider} and~\cite[Proposition~2.8]{ThomasZlapetal}; for countable $F$, this follows from $\GL_{n}(F)$ being countable and not finitely generated).

Due to work of Cornulier~\cite{Cornulier06}, a group $G$ has uncountable strong cofinality if and only if every orbit of every isometric action of $G$ on a metric space is bounded. In turn, this property entails a number of non-trivial geometric consequences (Remark~\ref{remark:cornulier}), directly leading to the following corollary of Theorem~\ref{theorem:bergman}.

\begin{cor}\label{corollary:fixed.point} Let $R$ be a non-discrete irreducible, continuous ring. Then every action of $\GL(R)$ by isometries on a non-empty complete $\mathrm{CAT}(0)$ space admits a fixed point. In particular, $\GL(R)$ has Serre's properties $(FH)$ and $(FA)$. \end{cor}
	
Moreover, as already indicated above, Theorem~\ref{theorem:simple} and Theorem~\ref{theorem:bergman} immediately imply the following result, a strong improvement of~\cite[Corollary~A]{Schneider25}.
	
\begin{cor}\label{corollary:bng} Let $R$ be a non-discrete irreducible, continuous ring. Then $\PGL(R)$ has bounded normal generation. \end{cor} 

This solves Problem~\ref{problem:carderi.thom}. At the same time, all of our results described above apply to any non-discrete irreducible, continuous ring, including, for instance, the ring of operators \emph{affiliated} with any $\mathrm{II}_{1}$ factor, as introduced by Murray and von Neumann~\cite{MurrayVonNeumann}. For further details on the latter family of rings, we refer to~\cite{VonNeumannBook,feldman,Berberian82,KadisonLiu,SchneiderIMRN}.

Our techniques have further applications concerning verbal width. Given a group $G$ and considering the countably generated free group $\free (\N)$, let us recall that any map $g \in G^{\N}$ naturally extends to a unique homomorphism $\free (\N) \to G, \, w \mapsto w(g)$ and define \begin{displaymath}
	w(G) \, \defeq \, \left\{ w(g) \left\vert \, g \in G^{\N} \right\}\! \right. \qquad (w \in \free(\N)) .
\end{displaymath} A group $G$ is called \emph{verbally simple} if, for every $w \in \free (\N)$, the subgroup\footnote{Note that $\langle w(G) \rangle$ is normal (in fact, fully characteristic) in $G$.} $\langle w(G) \rangle$ equals $G$ or $\{ 1\}$. Furthermore, a group $G$ is called \emph{verbally elliptic}~\cite[\S1.1, p.~2]{segal} if, for every $w \in \free (\N)$, there exists $m \in \N$ such that \begin{displaymath}
	\langle w(G) \rangle \, = \, \left( w(G) \cup w(G)^{-1} \right)^{m} .
\end{displaymath} Using the work of Larsen and Shalev~\cite{LarsenShalev08,LarsenShalev09}, we establish the following result, which generalizes~\cite[Corollary~1.2(C)]{BernardSchneider}, though with a weaker bound.

\begin{thm}\label{theorem:verbal.simplicity} Let $R$ be a non-discrete irreducible, continuous ring. If $w \in \free (\N)\setminus \{ 1 \}$, then $\GL(R) = w(\GL(R))^{64}$. In particular, $\GL(R)$ is verbally simple and verbally elliptic. \end{thm}

This paper is organized as follows. After providing some necessary algebraic prerequisites in Section~\ref{section:preliminaries}, we prove Theorem~\ref{theorem:simple} in Section~\ref{section:simple}. The following Section~\ref{section:bergman.abstract} revolves around the phenomenon of uncountable strong cofinality and an abstract method for establishing this property for groups of units in general rings. The latter is put into use in the subsequent Section~\ref{section:bergman.concrete}, where we prove Theorem~\ref{theorem:bergman} and deduce the corollaries advertised above. Our final Section~\ref{section:verbal.simplicity} contains the proof of Theorem~\ref{theorem:verbal.simplicity}.

\section{Preliminaries}\label{section:preliminaries}

In this section, we set up some general notation and terminology concerning groups and rings used throughout the entire manuscript.

To clarify some group-theoretic notation, let $G$ be a group. Given a subset $S \subseteq G$, we let $\langle S \rangle_{G}$ denote the subgroup of $G$ generated by $S$, and we define \begin{displaymath}
	\Cl_{G}(S) \, \defeq \, \left. \! \left\{ gsg^{-1} \, \right\vert s \in S, \, g \in G \right\} .
\end{displaymath} A subgroup $H_{0} \leq G$ is said to be \emph{normalized} by a subgroup $H_{1} \leq G$ if $gH_{0}g^{-1} = H_{0}$ for every $g \in H_{1}$. A \emph{commutator} in $G$ is any element of the form $[g,h] \defeq g^{-1}h^{-1}gh$ where $g,h \in G$. For subgroups $H_{0},H_{1} \leq G$, we define \begin{displaymath}
	[H_{0},H_{1}] \, \defeq \, \langle \{ [h_{0},h_{1}] \mid h_{0} \in H_{0}, \, h_{1} \in H_{1} \} \rangle_{G} .
\end{displaymath} Recall that both the \emph{commutator subgroup} $[G,G]$ and the \emph{center} \begin{displaymath}
	\cent (G) \, \defeq \, \{ h \in G \mid \forall g \in G \colon \, gh=hg \}
\end{displaymath} of $G$ constitute normal subgroups of $G$. 

In order to recollect some basic elements of ring theory, let $R$ be a unital ring. As usual, $R$ will be called \begin{itemize}
	\item[---\,] \emph{simple} if $\{0\}$ and $R$ are the only two-sided ideals of $R$, 
	\item[---\,] \emph{irreducible} if $R$ is non-zero and not isomorphic to the direct product of two non-zero rings. 
\end{itemize} Moreover, let us recall that the \emph{center} \begin{displaymath}
	\cent(R) \, \defeq \, \{ a \in R \mid \forall b \in R \colon \, ab = ba \}
\end{displaymath} is a unital subring of $R$ and that the set \begin{displaymath}
	\GL(R) \, \defeq \, \{ a \in R \mid \exists b \in R \colon \, ab = ba = 1 \}
\end{displaymath} of \emph{units} of $R$, equipped with the multiplication inherited from $R$, constitutes a group. We refer to the quotient \begin{displaymath}
	\PGL(R) \, \defeq \, \GL(R)/\cent(\GL(R))
\end{displaymath} as the \emph{projective unit group} of $R$. Given any $n \in \N$, we let $\M_{n}(R)$ denote the unital ring of $n \times n$ matrices with entries in $R$, and we define $\GL_{n}(R) \defeq \GL(\M_{n}(R))$. Furthermore, we recall that the set \begin{displaymath}
	\E(R) \, \defeq \, \{ e \in R \mid ee = e \}
\end{displaymath} of \emph{idempotent} elements of $R$ is partially ordered by the relation \begin{displaymath}
	e \leq f \ \, :\Longleftrightarrow \ \, ef = fe = e \qquad (e,f \in \E(R)) .
\end{displaymath} As is customary, two elements $e,f \in \E(R)$ will be called \emph{orthogonal} if $ef=fe=0$. We record the following well-known elementary fact for repeated later use.

\begin{remark}\label{remark:difference} Let $R$ be a unital ring. If $e,f \in \E(R)$ and $f \leq e$, then $e-f$ is an element of $\E(R)$ orthogonal to $f$. \end{remark}

If $R$ is a unital ring and $e \in \E(R)$, then the set $eRe$ constitutes a subring of $R$, with multiplicative unit $e$.

\begin{remark}\label{remark:matrix.units} Let $R$ be a unital ring, let $n \in \N$, and let $s \in \M_{n}(R)$ be a \emph{family of matrix units} for $R$, which means that $s_{11} + \ldots + s_{nn} = 1$ and \begin{displaymath}
	s_{ij}s_{k \ell} \, = \, \begin{cases}
				\, s_{i\ell} & \text{if } j = k , \\
				\, 0 & \text{otherwise}
			\end{cases}
\end{displaymath} for all $i,j,k, \ell \in \{ 1,\ldots,n \}$. Then, with $S \defeq s_{11}Rs_{11}$, the map \begin{displaymath}
	\M_{n}(S) \, \longrightarrow \, R , \quad a \, \longmapsto \, \sum\nolimits_{i,j=1}^{n} s_{i1}a_{ij}s_{1j}
\end{displaymath} is a ring isomorphism (see, e.g.,~\cite[II.III, Proof of Theorem~3.3, pp.~99--100]{VonNeumannBook}), thus restricts to a group isomorphism from $\GL_{n}(S) = \GL(\M_{n}(S))$ to $\GL(R)$. \end{remark}

A unital ring $R$ is called \emph{(von Neumann) regular}~\cite[II.II, Definition~2.2, p.~70]{VonNeumannBook} if \begin{displaymath}
	\forall a \in R \ \exists b \in R \colon \quad aba = a .
\end{displaymath} For every regular ring $R$, the set $\lat(R) \defeq \{ aR \mid a \in R\}$ of its principal right~ideals, partially ordered by inclusion, constitutes a complemented, modular lattice, as established by von Neumann~\cite[II.II, Theorem~2.4, p.~72]{VonNeumannBook}. A \emph{continuous ring} is a regular ring $R$ such that $\lat(R)$ is a \emph{continuous geometry}~\cite[Chapter~13, p.~160--161]{GoodearlBook}, that is, the lattice $\lat(R)$, moreover, is complete and satisfies \begin{displaymath}
	I \wedge \bigvee C \, = \, \bigvee \{ I \wedge J \mid J \in C\}, \qquad \, I \vee \bigwedge C \, = \, \bigwedge \{ I \vee J \mid J \in C\} 
\end{displaymath} for every chain $C \subseteq \lat(R)$ and every $I \in \lat(R)$.

Of course, every simple non-zero ring is irreducible. Within the class of continuous rings, the converse implication holds as well.

\begin{prop}[{\cite[Lemma~3.1]{Maeda50}}]\label{proposition:simplicity.vs.irreducibility} A continuous ring is irreducible if and only if it is simple and non-zero. \end{prop}

\begin{proof} Alternatively, proofs of this are to be found in~\cite[VII.3, Hilfssatz~3.1, p.~166]{MaedaBook} and~\cite[Corollary~13.26, p.~170]{GoodearlBook}. \end{proof}

A \emph{rank function} on a regular ring $R$ is a map $\rho \colon R \to [0,1]$ such that \begin{itemize}
	\item[---\,] $\rho(1) = 1$,
	\item[---\,] $\rho(ab) \leq \min \{ \rho(a), \rho(b)\}$ for all $a,b \in R$,
	\item[---\,] $\rho(e+f) = \rho(e) + \rho(f)$ for any two orthogonal $e,f \in \E(R)$,
	\item[---\,] $\rho(a) > 0$ for every $a \in R\setminus \{ 0 \}$.\footnote{The third condition already implies that $\rho(0) = 0$.}
\end{itemize} If $\rho$ is a rank function on a regular ring $R$, then \begin{displaymath}
	d_{\rho} \colon \, R \times R \, \longrightarrow \, [0,1], \quad (a,b) \, \longmapsto \, \rho(a-b)
\end{displaymath} is a metric on $R$ (see~\cite[II.XVIII, Lemma~18.1, p.~231]{VonNeumannBook}, \cite[VI.5, Satz~5.1, p.~154]{MaedaBook}, or~\cite[Proposition~19.1, p.~282]{GoodearlBook}).

\enlargethispage{1mm}

\begin{thm}[\cite{VonNeumannBook}]\label{theorem:unique.rank.function} Let $R$ be an irreducible, continuous ring. Then $R$ admits a unique rank function $\rk_{R} \colon R \to [0,1]$. Moreover, the metric $d_{R} \defeq d_{\rk_{R}}$ is complete. \end{thm}

\begin{proof} While the existence is due to~\cite[II.XVII, Theorem~17.1, p.~224]{VonNeumannBook} and the uniqueness is due to~\cite[II.XVII, Theorem~17.2, p.~226]{VonNeumannBook}, the completeness is established in~\cite[II.XVII, Theorem~17.4, p.~230]{VonNeumannBook}. \end{proof}

An irreducible, continuous ring $R$ will be called \emph{discrete} if its \emph{rank topology}, i.e., the topology on $R$ generated by the metric $d_{R}$, is discrete. We recollect several useful characterizations of non-discreteness, essentially consequences of von Neumann's work~\cite{VonNeumannBook}, in the following basic lemma.

\begin{lem}\label{lemma:nondiscrete} Let $R$ be an irreducible, continuous ring. The following are equivalent. \begin{enumerate}
	\item\label{lemma:nondiscrete.1} $R$ is non-discrete.
	\item\label{lemma:nondiscrete.2} $\rk_{R}(R) = [0,1]$.
	\item\label{lemma:nondiscrete.3} For every $n \in \N_{>0}$, there exists a family of matrix units in $\M_{n}(R)$ for $R$.
	\item\label{lemma:nondiscrete.4} $\E(R)$ contains an infinite set of pairwise orthogonal elements.
\end{enumerate} \end{lem}

\begin{proof} \ref{lemma:nondiscrete.1}$\Longleftrightarrow$\ref{lemma:nondiscrete.2}. This is established in~\cite[Remark~7.12]{SchneiderGAFA}.
	
\ref{lemma:nondiscrete.1}$\Longrightarrow$\ref{lemma:nondiscrete.3}$\wedge$\ref{lemma:nondiscrete.4}. Due to the Hausdorff maximal principle, there exists a maximal chain $E$ in $(\E(R),{\leq})$. According to~\cite[Corollary~7.20]{SchneiderGAFA}, since $R$ is non-discrete, ${\rk_{R}}\vert_{E} \colon (E,{\leq}) \to ([0,1],{\leq})$ is an isomorphism of linearly ordered sets. We continue with separate arguments for~\ref{lemma:nondiscrete.3} and~\ref{lemma:nondiscrete.4}. \begin{itemize}
	\item[\ref{lemma:nondiscrete.3}] Consider any $n \in \N_{>0}$. For each $i \in \{ 0,\ldots,n \}$, we define $f_{i} \defeq ({\rk_{R}}\vert_{E})^{-1}\!\left(\tfrac{i}{n} \right)$. Then $0 = f_{0} \leq \ldots \leq f_{n} = 1$. For each $i \in \{ 1,\ldots,n \}$, we know from Remark~\ref{remark:difference} that $e_{i} \defeq f_{i}-f_{i-1}$ is an element of $\E(R)$ orthogonal to $f_{i-1}$, so that $\rk_{R}(f_{i}) = \rk_{R}(e_{i}) + \rk_{R}(f_{i-1})$ and thus \begin{displaymath}
		\qquad \rk_{R}(e_{i}) \, = \, \rk_{R}(f_{i}) - \rk_{R}(f_{i-1}) \, = \, \tfrac{i}{n} - \tfrac{i-1}{n} \, = \, \tfrac{1}{n} .
	\end{displaymath} Of course, $\sum_{i=1}^{n} e_{i} = \sum_{i=1}^{n} f_{i} - \sum_{i=1}^{n} f_{i-1} = 1$. Moreover, if $i,j \in \{ 1,\ldots,n \}$ and $i<j$, then \begin{displaymath}
		\qquad e_{i}e_{j} \, = \, e_{i}f_{i}e_{j} \, = \, e_{i}f_{i}f_{j-1}e_{j} \, = \, 0 \, = \, e_{j}f_{j-1}f_{i}e_{i} \, = \, e_{j}f_{i}e_{i} \, = \, e_{j}e_{i} .
	\end{displaymath} That is, $e_{1},\ldots,e_{n}$ are pairwise orthogonal. By~\cite[Lemma~4.6]{BernardSchneider}, there exists a family of matrix units $s \in \M_{n}(R)$ for $R$ with $s_{ii} = e_{i}$ for each $i \in \{ 1,\ldots,n\}$.
	\item[\ref{lemma:nondiscrete.4}] Choose any strictly increasing sequence $t \in [0,1]^{\N}$. Then \begin{displaymath}
		\qquad f_{i} \, \defeq \, ({\rk_{R}}\vert_{E})^{-1}(t_{i}) \qquad (i \in \N)
	\end{displaymath} is a strictly increasing sequence in $(\E(R),{\leq})$. For each $i \in \N_{>0}$, Remark~\ref{remark:difference} asserts that $e_{i} \defeq f_{i}-f_{i-1}$ is an element of $\E(R)$ orthogonal to $f_{i-1}$. In turn, if $i,j \in \N_{>0}$ and $i<j$, then \begin{displaymath}
		\qquad e_{i}e_{j} \, = \, e_{i}f_{i}e_{j} \, = \, e_{i}f_{i}f_{j-1}e_{j} \, = \, 0 \, = \, e_{j}f_{j-1}f_{i}e_{i} \, = \, e_{j}f_{i}e_{i} \, = \, e_{j}e_{i} .
	\end{displaymath} Since $e_{i} \ne 0$ for each $i \in \N$, it follows that $(e_{i})_{i \in \N}$ are pairwise distinct.
\end{itemize}

\ref{lemma:nondiscrete.3}$\vee$\ref{lemma:nondiscrete.4}$\Longrightarrow$\ref{lemma:nondiscrete.1}. It suffices to infer that $\left\{ a \in R \left\vert \, d_{R}(a,0) \leq \tfrac{1}{n} \right\} \!\right. \ne \{ 0 \}$ for all $n \in \N_{>0}$. For this purpose, let $n \in \N_{>0}$. The disjunction of~\ref{lemma:nondiscrete.3} and~\ref{lemma:nondiscrete.4} implies the existence of pairwise orthogonal elements $e_{1},\ldots,e_{n} \in \E(R)\setminus \{ 0 \}$. For each $i \in \{ 1,\ldots,n-1 \}$, one readily checks that $\sum_{j=1}^{i} e_{j}$ is an element of $\E(R)$ orthogonal to $e_{i+1}$. Hence, \begin{displaymath}
	1 \, \geq \, \rk_{R}\!\left( \sum\nolimits_{j=1}^{n} e_{j}\right) \, = \, \rk_{R}\!\left( \sum\nolimits_{j=1}^{n-1} e_{j}\right)\! + \rk_{R}(e_{n}) \, = \, \ldots \, = \, \sum\nolimits_{j=1}^{n} \rk_{R}(e_{j})
\end{displaymath} and so $d_{R}(e_{j},0) = \rk_{R}(e_{j}) \leq \tfrac{1}{n}$ for some $j \in \{ 1,\ldots,n \}$, as desired. \end{proof}

\begin{remark}\label{remark:corner.rings} Let $R$ be a unital ring and let $e \in \E(R)$. \begin{enumerate}
	\item\label{remark:corner.rings.1} If $R$ is regular, then $eRe$ is regular by~\cite[II.II, Theorem~2.11, p.~77]{VonNeumannBook}.
	\item\label{remark:corner.rings.2} If $R$ is continuous, then $eRe$ is continuous by~\cite[Proposition~13.7, p.~162]{GoodearlBook}.
	\item\label{remark:corner.rings.3} Suppose that $R$ is continuous and irreducible and that $e \ne 0$. Then the continuous ring $eRe$ is irreducible with \begin{displaymath}
		\qquad \rk_{eRe} \, = \, \tfrac{1}{\rk_{R}(e)}{\rk_{R}}\vert_{eRe} .
	\end{displaymath} Moreover, if $R$ is non-discrete, then so is $eRe$. (See~\cite[Remark~4.8]{BernardSchneider} for details.)
\end{enumerate} \end{remark}

\section{Projective simplicity}\label{section:simple}

This section is devoted to the proof of Theorem~\ref{theorem:simple}. Among other things, our argument relies on the work of Bass~\cite{bass} and a recent result of~\cite{BernardSchneider}.

A unital ring $R$ is said to be \emph{unit-regular}~\cite[Chapter~4, p.~37]{GoodearlBook} if, for every $a \in R$, there exists $u \in \GL(R)$ such that $aua = a$. In~\cite[p.~172]{GoodearlBook}, the following result is attributed to Utumi~\cite[Theorems~5.1, 5.6]{utumi}.

\begin{prop}[{\cite[Corollary~13.23, p.~170]{GoodearlBook}}]\label{proposition:unit.regular} Every continuous ring is unit-regular. \end{prop}

Given $n \in \N_{>0}$, a unital ring $R$ is said to be of \emph{stable range at most $n$} (cf.~\cite{bass,vaserstein}) if, for all $a_{1},\ldots,a_{n+1} \in R$ such that $a_{1}R + \ldots + a_{n+1}R = R$, there exist $b_{1},\ldots,b_{n} \in R$ such that $(a_{1}+a_{n+1}b_{1})R + \ldots + (a_{n}+a_{n+1}b_{n})R = R$. This terminology is justified by a result of Vaserstein~\cite[Theorem~1]{vaserstein}, which states that any unital ring of stable range at most $n \in \N_{>0}$ is necessarily of stable range at most $n+1$. The \emph{stable range} of a unital ring $R$ is defined to be \begin{displaymath}
	\sr(R) \, \defeq \, \inf \{ n \in \N_{>0} \mid R \text{ of stable range at most } n \} \, \in \, {\N_{>0}} \cup {\{ \infty \}} . 
\end{displaymath} According to~\cite[Theorem~2]{vaserstein}, the stable range of a unital ring may be defined equivalently by an analogous condition in terms of principal left ideals (as opposed to principal right ideals). As noted in~\cite[p.~48]{GoodearlBook}, the following originates in independent works of Fuchs~\cite[Corollary~1, Theorem~4]{fuchs} and Kaplansky~\cite[Proposition~8]{henriksen}.

\begin{prop}[{\cite[Proposition~4.12, p.~41]{GoodearlBook}}]\label{proposition:stable.range} A regular ring $R$ is unit-regular if and only if $\sr(R) = 1$. \end{prop}

\begin{proof} In view of~\cite[Theorem~2.6]{vaserstein2}, this is precisely~\cite[Proposition~4.12, p.~41]{GoodearlBook}. \end{proof}

In order to clarify some further terminology and notation, let $R$ be a unital ring and let $n \in \N$. An element $a \in \M_{n}(R)$ is called an \emph{elementary transvection} if there exist $i,j \in \{ 1,\ldots,n \}$ such that \begin{itemize}
	\item[---\,] $a_{kk} = 1$ for each $k \in \{ 1,\ldots,n \}$, and
	\item[---\,] $a_{k\ell} = 0$ for all $k,\ell \in \{ 1,\ldots,n \}$ with $k \ne \ell$ and $(k,\ell) \ne (i,j)$.
\end{itemize} Of course, any elementary transvection belongs to $\GL_{n}(R)$. The subgroup of $\GL_{n}(R)$ generated by the set of all elementary transvections will be denoted by $\EL_{n}(R)$. Following Bass~\cite[Chapter~I, \S1, p.~492]{bass}, given a two-sided ideal $I$ of $R$, we consider the natural group homomorphism \begin{displaymath}
	\theta_{I} \colon \, \GL_{n}(R) \, \longrightarrow \, \GL_{n}(R/I), \quad (a_{ij}) \, \longmapsto \, (a_{ij} + I)
\end{displaymath} along with its kernel \begin{displaymath}
	\GL_{n}(R,I) \, \defeq \, \ker (\theta_{I}) \, \unlhd \, \GL_{n}(R),
\end{displaymath} and we let $\EL_{n}(R,I)$ denote the normal subgroup of $\EL_{n}(R)$ generated by all elementary transvections in $\GL_{n}(R,I)$.

\begin{thm}[{\cite[Theorem~(4.2), items e) and f)]{bass}}]\label{theorem:bass} Let $R$ be a unital ring, $n \in \N$. \begin{enumerate}
	\item\label{theorem:bass.a} Suppose that $\max\{\sr(R),2\} < n$. If a subgroup $G$ of $\GL_{n}(R)$ is normalized by $\EL_{n}(R)$, then there is a unique two-sided ideal $I$ of $R$ such that $\EL_{n}(R,I) \subseteq G$ and $\theta_{I}(G) \subseteq \cent(\GL_{n}(R/I))$.
	\item\label{theorem:bass.b} Suppose that $\max\{ 2\sr(R),3\} \leq n$. If $I$ is a two-sided ideal of $R$, then \begin{displaymath}
				\qquad \EL_{n}(R,I) \, = \, [\GL_{n}(R),\GL_{n}(R,I)] .
			\end{displaymath}
\end{enumerate} \end{thm}

Our argument proving Theorem~\ref{theorem:simple} moreover relies on the following result from~\cite{BernardSchneider}.

\begin{thm}[{\cite[Corollary~1.2(B)]{BernardSchneider}}]\label{theorem:commutator.subgroup} Let $R$ be a non-discrete irreducible, continuous ring. Then every element of $\GL(R)$ is a product of $7$ commutators. In particular, \begin{displaymath}
	[\GL(R),\GL(R)] \, = \, \GL(R) .
\end{displaymath} \end{thm}

Everything is prepared for the proof of our first main result.

\begin{proof}[Proof of Theorem~\ref{theorem:simple}] By Lemma~\ref{lemma:nondiscrete}, there exists a family of matrix units $s \in \M_{3}(R)$ for $R$. From Remark~\ref{remark:corner.rings}\ref{remark:corner.rings.3}, we know that $S \defeq s_{11}Rs_{11}$ is an irreducible, continuous ring. In particular, $S$ is unit-regular by Proposition~\ref{proposition:unit.regular}, whence $\sr(S) = 1$ due to Proposition~\ref{proposition:stable.range} and therefore \begin{equation}\label{simple.1}
	\EL_{3}(S) \, = \, \EL_{3}(S,S) \, \stackrel{\ref{theorem:bass}\ref{theorem:bass.b}}{=} \, [\GL_{3}(S),\GL_{3}(S,S)] \, = \, [\GL_{3}(S),\GL_{3}(S)] .
\end{equation} Moreover, as $[\GL(R),\GL(R)] = \GL(R)$ thanks to Theorem~\ref{theorem:commutator.subgroup} and $\GL_{3}(S) \cong \GL(R)$ by Remark~\ref{remark:matrix.units}, we see that \begin{equation}\label{simple.2}
	[\GL_{3}(S),\GL_{3}(S)] \, = \, \GL_{3}(S) .
\end{equation} The conjunction of~\eqref{simple.1} and~\eqref{simple.2} implies that \begin{equation}\label{simple.3}
	\EL_{3}(S) \, = \, \GL_{3}(S) .
\end{equation} By Proposition~\ref{proposition:simplicity.vs.irreducibility}, the ring $S$ is simple, i.e., $\{ 0 \}$ and $S$ are the only two-sided ideals of $S$. Hence, if $G$ is a normal subgroup of $\GL_{3}(S)$, then Theorem~\ref{theorem:bass}\ref{theorem:bass.a} asserts that \begin{itemize}
	\item[---\,] $\theta_{\{ 0 \}}(G) \subseteq \cent(\GL_{3}(S/\{ 0 \}))$ and thus $G \subseteq \cent(\GL_{3}(S))$, or
	\item[---\,] $\GL_{3}(S) \stackrel{\eqref{simple.3}}{=} \EL_{3}(S) = \EL_{3}(S,S) \subseteq G$.
\end{itemize} This shows that $\GL_{3}(S)/\cent(\GL_{3}(S))$ is simple. Since $\GL_{3}(S) \cong \GL(R)$ by Remark~\ref{remark:matrix.units}, we finally deduce that $\PGL(R) = \GL(R)/\cent(\GL(R))$ is simple. \end{proof}

\section{A general criterion}\label{section:bergman.abstract}

In this section, we first provide some background concerning uncountable strong cofinality, a phenomenon rooted in the work of Bergman~\cite{bergman} and further developed in~\cite{DrosteGoebel,DrosteHolland,Cornulier06,DrosteHollandUlbrich}, and then establish an abstract method for proving this property for groups of units in general rings (Theorem~\ref{theorem:abstract}).

Following the terminology of~\cite{DrosteGoebel}, we will say that a group $G$ has \begin{itemize}
	\item[---\,] the \emph{Bergman property}\footnote{In~\cite[Definition~2.3]{Cornulier06}, this property is called \emph{Cayley boundedness}.} if, for every generating set $S$ of $G$, there exists $n \in \N$ such that $G = \left(S \cup S^{-1} \cup \{ 1 \}\right)^{n}$,
	\item[---\,] \emph{uncountable cofinality} if, for every ascending chain $(H_{n})_{n \in \N}$ of subgroups of $G$ such that $\bigcup_{n \in \N} H_{n} = G$, there exists $n \in \N$ such that $G = H_{n}$,
	\item[---\,] \emph{uncountable strong cofinality}\footnote{By Theorem~\ref{theorem:bounded}, this property coincides with the notion of \emph{strong boundedness} in~\cite[Definition~2.1]{Cornulier06}.} if, every ascending chain of sets $(W_{n})_{n \in \N}$ such that $\bigcup\nolimits_{n \in \N} W_{n} = G$, there exist $m,n \in \N$ such that $G = W_{n}^{m}$.
\end{itemize}

\begin{remark}[see~{\cite[Proof of Theorem~2.2, (2)$\Longrightarrow$(3)]{rosendal}}]\label{remark:cofinality} A group $G$ has uncountable strong cofinality if and only if, for every sequence $(W_{n})_{n \in \N}$ of subsets of $G$ such that \begin{itemize}
	\item[---\,] $1 \in W_{n} = W_{n}^{-1}$ and $W_{n} \subseteq W_{n+1}$ for each $n \in \N$,
	\item[---\,] $\bigcup\nolimits_{n \in \N} W_{n} = G$,
\end{itemize} there exist $m,n \in \N$ such that $G = W_{n}^{m}$. \end{remark}

\begin{thm}[\cite{DrosteHolland,Cornulier06,rosendal}]\label{theorem:bounded} Let $G$ be a group. The following are equivalent. \begin{enumerate}
	\item\label{theorem:bounded.1} $G$ has uncountable strong cofinality.
	\item\label{theorem:bounded.2} $G$ has both uncountable cofinality and the Bergman property.
	\item\label{theorem:bounded.3} Every action of $G$ by isometries on a metric space has only bounded orbits.
\end{enumerate} \end{thm}

\begin{proof} The equivalence \ref{theorem:bounded.1}$\Longleftrightarrow$\ref{theorem:bounded.2} is established in~\cite[Theorem~2.2]{DrosteHolland} and~\cite[Proofs of Corollaries~2.1, 2.2]{DrosteHolland} (see also~\cite[Proposition~2.2(i)]{DrosteGoebel}), and the equivalence \ref{theorem:bounded.2}$\Longleftrightarrow$\ref{theorem:bounded.3} is due to~\cite[Proposition~2.4]{Cornulier06}, while \ref{theorem:bounded.1}$\Longleftrightarrow$\ref{theorem:bounded.3} is also proved in~\cite[Theorem~2.2]{rosendal}. \end{proof}

\begin{remark}[{\cite[Remark~2.2]{Cornulier06}}]\label{remark:cornulier} Let $G$ be a group with uncountable strong cofinality. Then every action of $G$ by isometries on a non-empty complete $\mathrm{CAT}(0)$ space admits a fixed point. In particular, $G$ has \begin{itemize}
	\item[---\,] \emph{Serre's property $(FH)$}, i.e., every action of $G$ by isometries on a real Hilbert space has a fixed point,
	\item[---\,] \emph{Serre's property $(FA)$}, i.e., every action of $G$ by automorphisms on a tree fixes a vertex or an edge.
\end{itemize} \end{remark}

\begin{remark}[{\cite[p.~433, paragraph before Lemma~7]{bergman}}]\label{remark:quotients} The class of groups with the Bergman property (resp., uncountable cofinality, uncountable strong cofinality) is closed under taking quotients. \end{remark}

A group $G$ is said to have \emph{bounded normal generation} if, for every $g \in G\setminus \{ 1\}$, there exists $n \in \N$ such that $G = \Cl_{G}\!\left( \left\{ g,g^{-1}\right\}\right)^{n}$. The following is well known.

\begin{lem}\label{lemma:bergman.simple} Every simple group with the Bergman property has bounded normal generation. \end{lem}
	
\begin{proof} Let $G$ be a simple group with the Bergman property. If $g \in G\setminus \{ 1\}$, then the subgroup of $G$ generated by $\Cl_{G}\!\left( \left\{ g,g^{-1}\right\}\right)$ is a non-trivial normal subgroup of $G$ and thus coincides with $G$ by simplicity, whence the Bergman property asserts the existence of some $n \in \N$ such that $G = \left(\Cl_{G}\!\left( \left\{ g,g^{-1}\right\}\right) \cup \{ 1\} \right)^{n} \subseteq \Cl_{G}\!\left( \left\{ g,g^{-1}\right\}\right)^{2n}$. \end{proof}	

We now turn to this section's main objective, which is a general sufficient criterion for uncountable strong cofinality for groups of units in unital rings (Theorem~\ref{theorem:abstract}). This involves the following construction.

\begin{lem}[{\cite[Lemma~5.2]{BernardSchneider}}]\label{lemma:gamma} Let $R$ be a unital ring and let $e \in \E(R)$. Then \begin{displaymath}
	\Gamma_{R}(e) \, \defeq \, \GL(eRe) + 1-e \, = \, \GL(R) \cap (eRe + 1-e)
\end{displaymath} is a subgroup of $\GL(R)$ and \begin{displaymath}
	\GL(eRe) \, \longrightarrow \, \Gamma_{R}(e),\quad a \, \longmapsto \, a+1-e
\end{displaymath} is a group isomorphism. Moreover, if $f \in \E(R)$ and $f \leq e$, then $\Gamma_{R}(f) \subseteq \Gamma_{R}(e)$. \end{lem}

\begin{definition} Let $R$ be a unital ring, $G$ be a subgroup of $\GL(R)$, and $e \in \E(R)$. Then we define \begin{displaymath}
	G(e) \, \defeq \, G \cap \Gamma_{R}(e) .
\end{displaymath} A subset $W \subseteq G$ is called \emph{full} (cf.~\cite[Definition~11.2]{BernardSchneider}) for $e$ if, for every $h \in G(e)$, there exists some $g \in W$ such that $eg = ge = he$. \end{definition}

The subsequent theorem generalizes~\cite[Theorem~3.1]{dowerk}, which itself provides a generalization of~\cite[Theorem~2.1]{DrosteHollandUlbrich}. 
Indeed, to recover \cite[Theorem~3.1]{dowerk}, one may just consider the unital ring of bounded linear operators on a Hilbert space. Our generalization will be suitable for the proof of Theorem~\ref{theorem:bergman}, but is applicable, in fact, to other classes of rings as well. These further applications of the following result will be the subject of forthcoming work.

\begin{thm}\label{theorem:abstract} Let $R$ be a unital ring and let $G$ be a subgroup of $\GL(R)$. Suppose that there exists a subset $E \subseteq \E(R)$ satisfying the following conditions: \begin{enumerate}[label=\textnormal{(\Alph{enumi})}]
	\item\label{cond1} There exists $(e_{n})_{n \in \N} \in E^{\N}$ such that, for every sequence $h_{n} \in G(e_{n})$ $(n \in \N)$, there exists $g \in G$ such that \begin{displaymath}
			\qquad \forall n \in \N \colon \quad e_{n}g = ge_{n} = h_{n}e_{n} .
		\end{displaymath}
	\item\label{cond2} For every $e \in E$, there exist a finite subset $F \subseteq G(e)$, $m \in \N$ and $f \in E$ with \begin{displaymath}
			\qquad G(f) \, \subseteq \, \Cl_{G(e)}(F)^{m} .
		\end{displaymath}
	\item\label{cond3} For every $e \in E$, there exist a finite subset $F \subseteq G$ and $\ell \in \N$ such that \begin{displaymath}
			\qquad G \, = \, (G(e) \cup F)^{\ell} .
		\end{displaymath}
\end{enumerate} Then $G$ has uncountable strong cofinality. \end{thm}

\begin{proof} Following Remark~\ref{remark:cofinality}, let $(W_{n})_{n \in \N}$ be any sequence of subsets of $G$ such that \begin{itemize}
	\item[---\,] $1 \in W_{n} = W_{n}^{-1}$ and $W_{n} \subseteq W_{n+1}$ for each $n \in \N$,
	\item[---\,] $\bigcup\nolimits_{n \in \N} W_{n} = G$.
\end{itemize} We establish the desired conclusion in three consecutive claims.
	
\textit{Claim 0}. There exist $n_{0} \in \N$ and $e \in E$ such that $W_{n_{0}}$ is full for $e$.
	
\textit{Proof of Claim 0}. The argument follows the lines of~\cite[Proof of Lemma~11.3(B)]{BernardSchneider}. Let $(e_{n})_{n \in \N}$ be as in~\ref{cond1}. For contradiction, suppose that, for every $n \in \N$, the set $W_{n}$ is not full for $e_{n}$. Then there exists a sequence $h_{n} \in G(e_{n})$ $(n \in \N)$ such that \begin{displaymath}
	\forall n \in \N \ \forall g \in W_{n} \colon \quad \neg (e_{n}g = ge_{n} = h_{n}e_{n}) .
\end{displaymath} According to~\ref{cond1}, we find some $g \in G$ such that $e_{n}g = ge_{n} = h_{n}e_{n}$ for each $n \in \N$. Hence, $g \notin \bigcup_{n \in \N} W_{n} = G$, which is the desired contradiction. Consequently, there exists $n_{0} \in \N$ such that $W_{n_{0}}$ is full for $e_{n_{0}} \eqdef e$. $\qed_{\text{Claim\! 0}}$
	
\textit{Claim 1}. There exist $m_{1},n_{1} \in \N$ and $f \in E$ such that $G(f) \subseteq W_{n_{1}}^{m_{1}}$.
	
\textit{Proof of Claim 1}. By~\ref{cond2}, there exist a finite set $F_{1} \subseteq G(e)$, $m \in \N$ and $f \in E$ such that $G(f) \subseteq \Cl_{G(e)}(F_{1})^{m}$. As $F_{1}$ is finite, we find $k_{1} \in \N$ such that $F_{1} \subseteq W_{k_{1}}$. Let $n_{1} \defeq \max\{ n_{0}, k_{1} \}$ and $m_{1} \defeq 3m$. We will show that $G(f) \subseteq W_{n_{1}}^{m_{1}}$. To this end, let $h \in G(f)$. Then there exist $g_{1},\ldots,g_{m} \in F_{1}$ and $t_{1},\ldots,t_{m} \in G(e)$ such that $h = t_{1}g_{1}t_{1}^{-1}\cdots t_{m}g_{m}t_{m}^{-1}$. Since $W_{n_{0}}$ is full for $e$, for each $i \in \{ 1,\ldots,m \}$ there exists $s_{i} \in W_{n_{0}}$ such that $es_{i} = s_{i}e = t_{i}e$, or equivalently, $s_{i}^{-1}e = es_{i}^{-1} = t_{i}^{-1}e$. Now, \begin{align*}
	&s_{1}g_{1}s_{1}^{-1}\cdots s_{m}g_{m}s_{m}^{-1}e \, = \, s_{1}eg_{1}es_{1}^{-1}e\cdots es_{m}eg_{m}es_{m}^{-1}e \\
	& \qquad = \, t_{1}eg_{1}et_{1}^{-1}e\cdots et_{m}eg_{m}et_{m}^{-1}e \, = \, t_{1}g_{1}t_{1}^{-1}\cdots t_{m}g_{m}t_{m}^{-1}e \, = \, he
\end{align*} and \begin{align*}
	&s_{1}g_{1}s_{1}^{-1}\cdots s_{m}g_{m}s_{m}^{-1}(1-e) \\
	&\qquad = \, s_{1}(1-e)g_{1}(1-e)s_{1}^{-1}(1-e)\cdots (1-e)s_{m}(1-e)g_{m}(1-e)s_{m}^{-1}(1-e) \\
	&\qquad = \, s_{1}(1-e)s_{1}^{-1}(1-e)\cdots (1-e)s_{m}(1-e)s_{m}^{-1}(1-e) \\
	&\qquad = \, s_{1}s_{1}^{-1}\cdots s_{m}s_{m}^{-1}(1-e) \, = \, 1-e .
\end{align*} Therefore, as $h \in G(f) \subseteq \Cl_{G(e)}(F_{1})^{m} \subseteq G(e)$, \begin{align*}
	h \, &= \, he + 1 - e \\
		&= \, s_{1}g_{1}s_{1}^{-1}\cdots s_{m}g_{m}s_{m}^{-1}e + s_{1}g_{1}s_{1}^{-1}\cdots s_{m}g_{m}s_{m}^{-1}(1-e) \\
		&= \, s_{1}g_{1}s_{1}^{-1}\cdots s_{m}g_{m}s_{m}^{-1} \, \in \, \! \left(W_{n_{0}}W_{k_{1}}W_{n_{0}}^{-1}\right)^{m} \, \subseteq \, W_{n_{1}}^{3m} \, = \, W_{n_{1}}^{m_{1}} . \hfill \qed_{\text{Claim\! 1}}
\end{align*} 
	
\textit{Claim 2}. There exist $m_{2},n_{2} \in \N$ such that $G = W_{n_{2}}^{m_{2}}$.
	
\textit{Proof of Claim 2}. Due to~\ref{cond3}, there exist a finite subset $F_{2} \subseteq G$ and $\ell \in \N$ such that $G = (G(f) \cup F_{2})^{\ell}$. Since $F_{2}$ is finite, we find some $k_{2} \in \N$ such that $F_{2} \subseteq W_{k_{2}}$. Let $n_{2} \defeq \max\{ n_{1},k_{2} \}$ and $m_{2} \defeq \ell \max\{ m_{1}, 1\}$. Then \begin{displaymath}
	G \, = \, (G(f) \cup F_{2})^{\ell} \, \subseteq \, \left( W_{n_{1}}^{m_{1}} \cup W_{k_{2}} \right)^{\ell} \, \subseteq \, W_{n_{2}}^{\ell \max\{ m_{1}, 1\}} \, = \, W_{n_{2}}^{m_{2}} . \hfill \qed_{\text{Claim\! 2}}
\end{displaymath}
	
This shows that $G$ has uncountable strong cofinality. \end{proof}

\enlargethispage{7mm}

\section{Uncountable strong cofinality}\label{section:bergman.concrete}

For the proof of Theorem~\ref{theorem:bergman}, we will verify the conditions specified in Theorem~\ref{theorem:abstract} for the unit group of a non-discrete irreducible, continuous ring, starting with~\ref{cond1}.

\begin{lem}\label{lemma:first.condition} Let $R$ be an irreducible, continuous ring. For every sequence $(e_{n})_{n \in \N}$ of pairwise orthogonal elements of $\E(R)$ and every sequence $h_{n} \in \Gamma_{R}(e_{n})$ $(n \in \N)$, there exists $g \in \GL(R)$ such that \begin{displaymath}
	\forall n \in \N \colon \quad e_{n}g = ge_{n} = h_{n}e_{n} .
\end{displaymath} \end{lem}

\begin{proof} Consider a sequence $(e_{n})_{n \in \N}$ of pairwise orthogonal elements of $\E(R)$ along with a sequence $h_{n} \in \Gamma_{R}(e_{n})$ $(n \in \N)$. Since $h_{n}e_{n} \in \GL(e_{n}Re_{n})$ for each $n \in \N$ by Lemma~\ref{lemma:gamma} and $(R,{d_{R}})$ is complete by Theorem~\ref{theorem:unique.rank.function}, thus~\cite[Lemma~11.3(A)]{BernardSchneider} asserts the existence of some $g \in \GL(R)$ such that $e_{n}g = ge_{n} = h_{n}e_{n}$ for every $n \in \N$. \end{proof}

We move on to establishing~\ref{cond2}. Given a unital ring $R$, we let $\cha(R)$ denote its \emph{characteristic}, and we consider $\I(R) \defeq \left\{ a \in R \left\vert \, a^{2} = 1 \right\} \! \right. \subseteq \GL(R)$.

\begin{lem}\label{lemma:existence.involution} Let $R$ be a non-discrete irreducible, continuous ring and let \begin{displaymath}
	\tau \, \defeq \, \begin{cases}
			\, \tfrac{1}{2} & \text{if } \cha(R)=2 , \\
			\, 1 & \text{otherwise}.
		\end{cases}
\end{displaymath} For every $t \in [0,\tau]$, there exists $g \in \I(R)$ such that $\rk_{R}(1-g) = t$. \end{lem}

\begin{proof} Let $t \in [0,\tau]$. Since $R$ is non-discrete, Lemma~\ref{lemma:nondiscrete} asserts the existence of some $a \in R$ with $\rk_{R}(a) = t$. As $R$ is regular, we find $b \in R$ such that $aba = a$. It follows that $e \defeq ab \in \E(R)$ and $eR = aR$. The latter entails that $\rk_{R}(e) = \rk_{R}(a) = t$. We now proceed by case analysis.
	
\textit{Case 1}: $\cha(R) = 2$. Note that $1-e \in \E(R)$ is orthogonal to $e$ by Remark~\ref{remark:difference}. Hence, $1 = \rk_{R}(1) = \rk_{R}(1-e) + \rk_{R}(e)$ and therefore \begin{displaymath}
	\rk_{R}(1-e) \, = \, 1-\rk_{R}(e) \, = \, 1-t \, \geq \, 1-\tfrac{1}{2} \, = \, \tfrac{1}{2} \, \geq \, t .
\end{displaymath} Thus, we may apply~\cite[Lemma~4.7(A)]{BernardSchneider} to find $f \in \E(R)$ such that $f \leq 1-e$ and $\rk_{R}(f) = t$. Due to~\cite[Lemma~9]{Ehrlich56}\footnote{The standing assumption of~\cite{Ehrlich56}, namely $\cha(R) \ne 2$, is not used in this particular argument.}, there exists $u \in \GL(R)$ such that $ueu^{-1} = f$. Considering $x \defeq ue$, we observe that $x^{2} = ueue = uefu = ue(1-f)fu = 0$. Since $\cha(R) = 2$, this entails that $g \defeq 1+x \in \I(R)$ according to~\cite[Remark~6.1(B)]{BernardSchneider}. Furthermore, as desired, \begin{displaymath}
	\rk_{R}(1-g) \, = \, \rk_{R}(-x) \, = \, \rk_{R}(x) \, = \, \rk_{R}(ue) \, = \, \rk_{R}(e) \, = \, t .
\end{displaymath}
	
\textit{Case 2}: $\cha(R) \ne 2$. Then $g \defeq 1-2e \in \I(R)$ according to~\cite[Lemma~1]{Ehrlich56} (see also~\cite[Remark~6.1(A)]{BernardSchneider}). Moreover, as intended, \begin{displaymath}
	\rk_{R}(1-g) \, = \, \rk_{R}(2e) \, \stackrel{2 \in \GL(R)}{=} \, \rk_{R}(e) \, = \, t. \qedhere
\end{displaymath} \end{proof}

\begin{lem}\label{lemma:bounded.normal.generation} Let $R$ be a non-discrete irreducible, continuous ring and let \begin{displaymath}
	\tau \, \defeq \, \begin{cases}
			\, \tfrac{1}{4} & \text{if } \cha(R)=2 , \\
			\, \tfrac{1}{2} & \text{otherwise}.
		\end{cases}
\end{displaymath} Let $g \in \I(R)$ with $\rk_{R}(1-g) = \tau$. Then $\GL(R) = \Cl_{\GL(R)}(g)^{32}$. \end{lem}

\begin{proof} From~\cite[Proposition 6.4]{BernardSchneider}, we know that \begin{displaymath}
	\Cl_{\GL(R)}(g) \, = \, \{ h \in \I(R) \mid \rk_{R}(1-h) = \tau \} .
\end{displaymath} Hence, \cite[Lemma~6.5]{BernardSchneider} and~\cite[Lemma~6.6]{BernardSchneider} together assert that \begin{displaymath}
	\I(R) \, \subseteq \, \Cl_{\GL(R)}(g)^{2} .
\end{displaymath} Since $\GL(R) = \I(R)^{16}$ by~\cite[Corollary~1.2(A)]{BernardSchneider}, we thus conclude that \begin{displaymath}
	\GL(R) \, = \, \Cl_{\GL(R)}(g)^{32} .\qedhere
\end{displaymath} \end{proof}

\begin{lem}\label{lemma:second.condition} Let $R$ be a non-discrete irreducible, continuous ring. Then, for every $e \in \E(R)\setminus \{ 0 \}$, there exists $g \in \Gamma_{R}(e)$ such that $\Gamma_{R}(e) = \Cl_{\Gamma_{R}(e)}(g)^{32}$. \end{lem}

\begin{proof} Let $e \in \E(R)\setminus \{ 0 \}$. Then $eRe$ is a non-discrete irreducible, continuous ring according to Remark~\ref{remark:corner.rings}\ref{remark:corner.rings.3}. Now, let\footnote{As $\cent(eRe) = \cent(R)e \cong \cent(R)$ by~\cite[Lemma~2.1]{Halperin62}, we have $\cha(eRe) = \cha(R)$.} \begin{displaymath}
	\tau \, \defeq \, \begin{cases}
			\, \tfrac{1}{4} & \text{if } \cha(eRe)=2 , \\
			\, \tfrac{1}{2} & \text{otherwise}.
		\end{cases}
\end{displaymath} According to Lemma~\ref{lemma:existence.involution}, there exists $g \in \I(eRe)$ such that $\rk_{eRe}(e-g) = \tau$. By Lemma~\ref{lemma:bounded.normal.generation}, it follows that $\GL(eRe) = \Cl_{\GL(eRe)}(g)^{32}$. Since \begin{displaymath}
	\iota \colon \, \GL(eRe) \, \longrightarrow \, \Gamma_{R}(e), \quad a \, \longmapsto \, a + 1 - e
\end{displaymath} is a group isomorphism by Lemma~\ref{lemma:gamma}, we conclude that $\Gamma_{R}(e) = \Cl_{\Gamma_{R}(e)}(\iota(g))^{32}$. \end{proof}

Finally, we proceed to~\ref{cond3}. The present author has learned the following fact from Cornulier's expository notes~\cite{Cornulier07} on the works of Shalom~\cite{shalom} and Kassabov~\cite{kassabov}. 

\begin{lem}[cf.~{\cite[Proposition~5]{Cornulier07}}\footnote{The statement of~\cite[Proposition~5]{Cornulier07} contains the negligible hypothesis of $R$ being finitely generated.}]\label{lemma:cornulier} Let $R$ be a unital ring and let $n \in \N$ be such that $\sr(R) \leq n$. Consider \begin{displaymath}
	G \, \defeq \, \left. \! \left\{ \begin{pmatrix} a & 0 \\ 0 & 1 \end{pmatrix} \, \right\vert a \in \GL_{n}(R) \right\} \, \subseteq \, \GL_{n+1}(R)
\end{displaymath} as well as the subsets \begin{displaymath}
	H_{1} \defeq \left. \! \left\{ \begin{pmatrix} 1_{n} & x \\ 0 & 1 \end{pmatrix} \, \right\vert x \in R^{n \times 1} \right\}, \quad H_{2} \defeq \left. \! \left\{ \begin{pmatrix} 1_{n} & 0 \\ x & 1 \end{pmatrix} \, \right\vert x \in R^{1 \times n} \right\}
\end{displaymath} of $\EL_{n+1}(R)$. Then \begin{displaymath}
	\GL_{n+1}(R) \, = \, GH_{1}H_{2}H_{1}H_{2} .
\end{displaymath} \end{lem}

\begin{proof} We reproduce the argument from~\cite[Proof of Proposition~5]{Cornulier07}. Let $a \in \GL_{n+1}(R)$. In particular, since $a$ is invertible, \begin{displaymath}
	a_{n+1,1}R + \ldots + a_{n+1,n+1}R \, = \, R .
\end{displaymath} Hence, as $\sr(R) \leq n$, we find $x_{1},\ldots,x_{n} \in R$ such that \begin{displaymath}
	(a_{n+1,1}+a_{n+1,n+1}x_{1})R + \ldots + (a_{n+1,n}+a_{n+1,n+1}x_{n})R \, = \, R .
\end{displaymath} In turn, there exist $y_{1},\ldots,y_{n} \in R$ such that \begin{displaymath}
	\sum\nolimits_{i=1}^{n} (a_{n+1,i}+a_{n+1,n+1}x_{i})y_{i} \, = \, 1-a_{n+1,n+1} .
\end{displaymath} Letting \begin{align*}
	&x \defeq \begin{pmatrix} x_{1} & \cdots & x_{n} \end{pmatrix} \in R^{1 \times n}, \qquad z \defeq \begin{pmatrix} y_{1} \\ \vdots \\ y_{n} \end{pmatrix} \in R^{n \times 1} , \\
	&s \defeq \begin{pmatrix} a_{n+1,1}+a_{n+1,n+1}x_{1} & \cdots & a_{n+1,n}+a_{n+1,n+1}x_{n} \end{pmatrix} \in R^{1 \times n} ,
\end{align*} we see that \begin{displaymath}
	\GL_{n+1}(R) \, \ni \, a \begin{pmatrix} 1_{n} & 0 \\ x & 1 \end{pmatrix}\begin{pmatrix} 1_{n} & y \\ 0 & 1 \end{pmatrix} \, = \, \begin{pmatrix} b & t \\ s & 1 \end{pmatrix}
\end{displaymath} for suitable $b \in \M_{n}(R)$ and $t \in R^{n \times 1}$. Since \begin{displaymath}
	\GL_{n+1}(R) \, \ni \, \begin{pmatrix} 1_{n} & -t \\ 0 & 1 \end{pmatrix} \begin{pmatrix} b & t \\ s & 1 \end{pmatrix} \begin{pmatrix} 1_{n} & 0 \\ -s & 1 \end{pmatrix} \, = \, \begin{pmatrix} b-ts & 0 \\ 0 & 1 \end{pmatrix} \, \in \, G ,
\end{displaymath} we conclude that \begin{displaymath}
	a \, = \, \begin{pmatrix} 1_{n} & t \\ 0 & 1 \end{pmatrix}\begin{pmatrix} b-ts & 0 \\ 0 & 1 \end{pmatrix}\begin{pmatrix} 1_{n} & 0 \\ s & 1 \end{pmatrix} \begin{pmatrix} 1_{n} & -y \\ 0 & 1 \end{pmatrix} \begin{pmatrix} 1_{n} & 0 \\ -x & 1 \end{pmatrix} \, \in \, H_{1}GH_{2}H_{1}H_{2} .
\end{displaymath} This shows that $\GL_{n+1}(R) = H_{1}GH_{2}H_{1}H_{2}$. Since \begin{displaymath}
	\begin{pmatrix} a & 0 \\ 0 & 1 \end{pmatrix} \begin{pmatrix} 1 & x \\ 0 & 1 \end{pmatrix} {\begin{pmatrix} a & 0 \\ 0 & 1 \end{pmatrix}}^{-1}\! \, = \, \begin{pmatrix} a & 0 \\ 0 & 1 \end{pmatrix} \begin{pmatrix} 1 & x \\ 0 & 1 \end{pmatrix} \begin{pmatrix} a^{-1} & 0 \\ 0 & 1 \end{pmatrix} \, = \, \begin{pmatrix} 1 & ax \\ 0 & 1 \end{pmatrix} \, \in \, H_{1}
\end{displaymath} for all $a \in \GL_{n}(R)$ and $x \in R^{n \times 1}$, we know that $H_{1}$ is normalized by $G$, thus \begin{displaymath}
	\GL_{n+1}(R) \, = \, H_{1}GH_{2}H_{1}H_{2} \, = \, GH_{1}H_{2}H_{1}H_{2} . \qedhere
\end{displaymath} \end{proof}

\enlargethispage{3mm}

\begin{lem}\label{lemma:shift} Let $n \in \N_{> 1}$ and let $R$ be a unital ring. Consider \begin{align*}
	H_{1} &\defeq \left. \! \left\{ \begin{pmatrix} 1_{n} & x \\ 0 & 1 \end{pmatrix} \, \right\vert x \in R^{n \times 1} \right\}, \quad H_{2} \defeq \left. \! \left\{ \begin{pmatrix} 1_{n} & 0 \\ x & 1 \end{pmatrix} \, \right\vert x \in R^{1 \times n} \right\} , \\
	G &\defeq \left. \! \left\{ \begin{pmatrix} a & 0 \\ 0 & 1 \end{pmatrix} \, \right\vert a \in \GL_{n}(R) \right\}, \quad c \defeq
		\begin{pmatrix} 0 & 1 \\
			1_{n} & 0	\end{pmatrix} 
\end{align*} in $\GL_{n+1}(R)$. Then \begin{displaymath}
	GH_{1}H_{2}H_{1}H_{2} \, \subseteq \, \left( G \cup \left\{ c,c^{-1}\right\}\right)^{16}.
\end{displaymath} \end{lem}

\begin{proof} We first argue that $H_{1} \subseteq c^{-1}Gc^{2}Gc^{-1}$: indeed, if $x \in R^{n \times 1}$, then we let \begin{displaymath}
	y \defeq \begin{pmatrix} x_{1} \\ \vdots \\ x_{n-1} \end{pmatrix} \in R^{(n-1) \times 1}, \qquad z \defeq \begin{pmatrix} 0 \\ \vdots \\ 0 \\ x_{n} \end{pmatrix} \in R^{(n-1) \times 1}
\end{displaymath} and observe that \begin{align*}
	\begin{pmatrix} 1_{n} & x \\ 0 & 1 \end{pmatrix} \, &= \, \begin{pmatrix} 1_{n-1} & 0 & y \\ 0 & 1 & 0 \\ 0 & 0 & 1 \end{pmatrix} \begin{pmatrix} 1 & 0 & 0 \\ 0 & 1_{n-1} & z \\ 0 & 0 & 1 \end{pmatrix} \\
	&= \, c^{-1}\begin{pmatrix} 1 & 0 & 0 \\ y & 1_{n-1} & 0 \\ 0 & 0 & 1 \end{pmatrix}c^{2}\begin{pmatrix} 1_{n-1} & z & 0 \\ 0 & 1 & 0 \\ 0 & 0 & 1 \end{pmatrix} c^{-1} \, \in \, c^{-1}Gc^{2}Gc^{-1} .
\end{align*} By symmetry, we see that $H_{2} \subseteq cGc^{-2}Gc$: explicitly, if $x \in R^{1 \times n}$, then we consider \begin{displaymath}
	y \defeq \begin{pmatrix} x_{1} & \cdots & x_{n-1} \end{pmatrix} \in R^{1 \times (n-1)}, \qquad z \defeq \begin{pmatrix} 0 & \cdots & 0 & x_{n} \end{pmatrix} \in R^{1 \times (n-1)}
\end{displaymath} and notice that \begin{align*}
	\begin{pmatrix} 1_{n} & 0 \\ x & 1 \end{pmatrix} \, &= \, \begin{pmatrix} 1 & 0 & 0 \\ 0 & 1_{n-1} & 0 \\ 0 & z & 1 \end{pmatrix}\begin{pmatrix} 1_{n-1} & 0 & 0 \\ 0 & 1 & 0 \\ y & 0 & 1 \end{pmatrix} \\
	&= \, c\begin{pmatrix} 1_{n-1} & 0 & 0 \\ z & 1 & 0 \\ 0 & 0 & 1 \end{pmatrix} c^{-2}\begin{pmatrix} 1 & y & 0 \\ 0 & 1_{n-1} & 0 \\ 0 & 0 & 1 \end{pmatrix}c \, \in \, cGc^{-2}Gc .
\end{align*} Consequently, \begin{align*}
	H_{1}H_{2}H_{1}H_{2} \, &\subseteq \, c^{-1}Gc^{2}Gc^{-1}cGc^{-2}Gcc^{-1}Gc^{2}Gc^{-1}cGc^{-2}Gc \\
	& = \, c^{-1}Gc^{2}Gc^{-2}Gc^{2}Gc^{-2}Gc
\end{align*} and therefore \begin{displaymath}
	GH_{1}H_{2}H_{1}H_{2} \, \subseteq \, Gc^{-1}Gc^{2}Gc^{-2}Gc^{2}Gc^{-2}Gc \, \subseteq \, \left( G \cup \left\{ c,c^{-1}\right\}\right)^{16}. \qedhere
\end{displaymath} \end{proof}

\begin{lem}\label{lemma:local.generation.basic} Let $n \in \N_{>1}$, let $R$ be a unital ring of stable range at most $n$, and let \begin{displaymath}
	G \defeq \left. \! \left\{ \begin{pmatrix} a & 0 \\ 0 & 1 \end{pmatrix} \, \right\vert a \in \GL_{n}(R) \right\} \subseteq \GL_{n+1}(R), \quad c \defeq \begin{pmatrix} 0 & 1 \\ 1_{n} & 0	\end{pmatrix} \in \GL_{n+1}(R) .
\end{displaymath} Then $\GL_{n+1}(R) = \left( G \cup \left\{ c,c^{-1}\right\}\right)^{16}$. \end{lem}

\begin{proof} This follows from Lemma~\ref{lemma:cornulier} and Lemma~\ref{lemma:shift}. \end{proof}

\begin{lem}\label{lemma:local.generation} Let $R$ be an irreducible, continuous ring, and let $e,f \in \E(R)$ be such that $e \leq f$ and $\rk_{R}(e) = \tfrac{2}{3}\rk_{R}(f)$. Then there exists $g \in \Gamma_{R}(f)$ such that \begin{displaymath}
	\Gamma_{R}(f) \, = \, \left(\Gamma_{R}(e) \cup \left\{ g,g^{-1}\right\}\right)^{16} .
\end{displaymath} \end{lem}

\begin{proof} Since $\Gamma_{R}(0) = \{ 1 \}$, the desired conclusion is trivial for $f = 0$. Thus, we assume that $f \ne 0$, whence $S \defeq fRf$ is an irreducible, continuous ring with $\rk_{S} = \tfrac{1}{\rk_{R}(f)}{\rk_{R}}\vert_{S}$ by Remark~\ref{remark:corner.rings}\ref{remark:corner.rings.3}. Appealing to Remark~\ref{remark:difference}, we note that $e_{3} \defeq f-e \in \E(S)$ is orthogonal to $e$, wherefore $\rk_{S}(f) = \rk_{S}(e_{3}) + \rk_{S}(e)$ and so \begin{displaymath}
	\tfrac{1}{3} \, = \, \rk_{S}(f) - \rk_{S}(e) \, = \, \rk_{S}(e_{3}) \, \in \, \rk_{S}(S) .
\end{displaymath} Hence, by~\cite[Lemma~4.7(A)]{BernardSchneider}, there exists $e_{1} \in \E(S)$ such that $e_{1} \leq e$ and $\rk_{S}(e_{1}) = \tfrac{1}{3}$. Using Remark~\ref{remark:difference}, we observe that $e_{2} \defeq e-e_{1} \in \E(S)$ is orthogonal to $e_{1}$, wherefore $\rk_{S}(e) = \rk_{S}(e_{2}) + \rk_{S}(e_{1})$ and thus $\rk_{S}(e_{2}) = \rk_{S}(e)-\rk_{S}(e_{1}) = \tfrac{1}{3}$. Straightforward calculations show that $e_{1},e_{2},e_{3}$ are pairwise orthogonal. Moreover, $e_{1} + e_{2} + e_{3} = f$. Consequently, due to~\cite[Lemma~4.6]{BernardSchneider}, there exists a family of matrix units $s \in \M_{3}(S)$ for $S$ with $s_{ii} = e_{i}$ for each $i \in \{ 1,2,3 \}$. In turn, for $T \defeq e_{1}Se_{1}$, the map \begin{displaymath}
	\phi \colon \, \M_{3}(T) \, \longrightarrow \, S , \quad a \, \longmapsto \, \sum\nolimits_{i,j=1}^{3} s_{i1}a_{ij}s_{1j}
\end{displaymath} is a (necessarily unital) ring isomorphism by Remark~\ref{remark:matrix.units}, hence restricts to a group isomorphism from $\GL_{3}(T) = \GL(\M_{3}(T))$ to $\GL(S)$. Note that $T$ is a continuous ring by Remark~\ref{remark:corner.rings}\ref{remark:corner.rings.2}, thus unit-regular by Proposition~\ref{proposition:unit.regular} and therefore of stable range~$1$ due to Proposition~\ref{proposition:stable.range}. Concerning \begin{displaymath}
	G \defeq \left. \! \left\{ \begin{pmatrix} a & 0 \\ 0 & 1 \end{pmatrix} \, \right\vert a \in \GL_{2}(T) \right\} \subseteq \GL_{3}(T), \qquad 
	c \defeq \begin{pmatrix} 0 & 0 & 1 \\ 1 & 0 & 0	\\ 0 & 1 & 0 \end{pmatrix} \in \GL_{3}(T) ,
\end{displaymath} thus Lemma~\ref{lemma:local.generation.basic} asserts that $\GL_{3}(T) = \left( G \cup \left\{ c,c^{-1}\right\}\right)^{16}$. As $G = \Gamma_{\M_{3}(T)}(\tilde{e})$ for \begin{displaymath}
	\tilde{e} \defeq \begin{pmatrix} 1 & 0 & 0 \\ 0 & 1 & 0 \\ 0 & 0 & 0 \end{pmatrix} \in \E(\M_{3}(T)) 
\end{displaymath} and therefore $\phi(G) = \phi(\Gamma_{\M_{3}(T)}(\tilde{e})) = \Gamma_{S}(\phi(\tilde{e})) = \Gamma_{S}(e)$, we infer that \begin{align*}
	\GL(S) \, &= \, \phi(\GL_{3}(T)) \, = \, \phi\!\left( \left( G \cup \left\{ c,c^{-1}\right\}\right)^{16} \right) \\
		& = \, \left( \phi(G) \cup \left\{ \phi(c),\phi(c)^{-1}\right\}\right)^{16} \, = \, \left( \Gamma_{S}(e) \cup \left\{ \phi(c),\phi(c)^{-1}\right\}\right)^{16} .
\end{align*} Since $\iota \colon \GL(S) \to \Gamma_{R}(f), \, a \mapsto a +1-f$ is a group isomorphism by Lemma~\ref{lemma:gamma} and \begin{displaymath}
	\iota(\Gamma_{S}(e)) \, = \, \Gamma_{S}(e) + 1-f \, = \, \GL(eSe) + f-e + 1-f \, = \, \GL(eRe) + 1-e \, = \, \Gamma_{R}(e) ,
\end{displaymath} we conclude that \begin{align*}
	\Gamma_{R}(f) \, &= \, \iota(\GL(S)) \, = \, \iota\!\left( \left( \Gamma_{S}(e) \cup \left\{ \phi(c),\phi(c)^{-1}\right\}\right)^{16} \right) \\
		& = \, \left( \iota(\Gamma_{S}(e)) \cup \left\{ \iota(\phi(c)),\iota(\phi(c))^{-1}\right\}\right)^{16} \, = \, \left( \Gamma_{R}(e) \cup \left\{ g,g^{-1}\right\}\right)^{16}
\end{align*} for $g \defeq \iota(\phi(c)) \in \Gamma_{R}(f)$, as intended. \end{proof}

The proof of the following lemma is inspired by the arguments in~\cite[Proof of Corollary~3.6]{DrosteHollandUlbrich} and~\cite[Proof of Theorem~1.1]{dowerk}.

\begin{lem}\label{lemma:third.condition} Let $R$ be a non-discrete irreducible, continuous ring. Then, for every $e \in \E(R)\setminus \{ 0 \}$, there exist a finite subset $F \subseteq \GL(R)$ and $\ell \in \N$ such that
\begin{displaymath}
	\GL(R) \, = \, (\Gamma_{R}(e) \cup F)^{\ell} .
\end{displaymath} \end{lem}

\begin{proof} Let $e \in \E(R)\setminus \{ 0 \}$. Choose any $n \in \N$ such that $\left(\tfrac{2}{3}\right)^{n}\! \leq \rk_{R}(e)$. Since $R$ is non-discrete, Lemma~\ref{lemma:nondiscrete} asserts that $\rk_{R}(R) = [0,1]$. Thus, by an $n$-fold application of~\cite[Lemma~4.7(A)]{BernardSchneider}, we find $e_{n},\ldots,e_{1} \in \E(R)$ such that \begin{enumerate}
	\item[---\,] $e_{n} \leq e$,
	\item[---\,] $e_{n} \leq \ldots \leq e_{1}$,
	\item[---\,] $\rk_{R}(e_{i}) = \left(\tfrac{2}{3}\right)^{i}$ for each $i \in \{ 1,\ldots,n \}$.
\end{enumerate} Furthermore, let $e_{0} \defeq 1$. For each $i \in \{ 1,\ldots,n \}$, since $\rk_{R}(e_{i}) = \tfrac{2}{3}\rk_{R}(e_{i-1})$, our Lemma~\ref{lemma:local.generation} asserts the existence of some $g_{i} \in \Gamma_{R}(e_{i-1})$ such that \begin{displaymath}
	\Gamma_{R}(e_{i-1}) \, = \, \left(\Gamma_{R}(e_{i}) \cup \left\{ g_{i},g_{i}^{-1}\right\}\right)^{16} .
\end{displaymath} Hence, considering $\ell \defeq 16^{n}$ and $F \defeq \left\{ g_{1},\ldots,g_{n},g_{1}^{-1},\ldots,g_{n}^{-1}\right\}$, we infer that \begin{align*}
	\GL(R) \, &= \, \Gamma_{R}(e_{0}) \, = \, \left(\Gamma_{R}(e_{1}) \cup \left\{ g_{1},g_{1}^{-1}\right\}\right)^{16} \, \subseteq \, \left(\Gamma_{R}(e_{2}) \cup \left\{ g_{1},g_{2},g_{1}^{-1},g_{2}^{-1}\right\}\right)^{16^{2}} \\
		&\subseteq \, \ldots \, \subseteq \, \left(\Gamma_{R}(e_{n}) \cup \left\{ g_{1},\ldots,g_{n},g_{1}^{-1},\ldots,g_{n}^{-1}\right\}\right)^{16^{n}} \, \stackrel{\ref{lemma:gamma}}{\subseteq} \, (\Gamma_{R}(e) \cup F)^{\ell} . \qedhere
\end{align*} \end{proof}

Everything is prepared for the proofs of the advertised results.

\begin{proof}[Proof of Theorem~\ref{theorem:bergman}] We observe that the hypotheses of Theorem~\ref{theorem:abstract} are satisfied for $E \defeq \E(R)\setminus \{ 0 \}$: while~\ref{cond1} follows from Lemma~\ref{lemma:nondiscrete} and Lemma~\ref{lemma:first.condition}, condition~\ref{cond2} is due to Lemma~\ref{lemma:second.condition}, and condition~\ref{cond3} is established in Lemma~\ref{lemma:third.condition}. Hence, $\GL(R)$ has uncountable strong cofinality by Theorem~\ref{theorem:abstract}. \end{proof}

\begin{proof}[Proof of Corollary~\ref{corollary:fixed.point}] This follows immediately from Theorem~\ref{theorem:bergman}, Theorem~\ref{theorem:bounded}, and Remark~\ref{remark:cornulier}. \end{proof}

\begin{proof}[Proof of Corollary~\ref{corollary:bng}] According to Theorem~\ref{theorem:bergman} and Remark~\ref{remark:quotients}, the quotient group $\PGL(R) = \GL(R)/\cent(\GL(R))$ has the Bergman property. Since $\PGL(R)$ is simple by Theorem~\ref{theorem:simple}, thus $\PGL(R)$ has bounded normal generation by Lemma~\ref{lemma:bergman.simple}. \end{proof}

\section{Verbal simplicity}\label{section:verbal.simplicity}

We conclude this note with the proof of Theorem~\ref{theorem:verbal.simplicity}, which involves an application of the following result about finite alternating groups established by Larsen and Shalev in both~\cite{LarsenShalev08} and~\cite{LarsenShalev09} via different methods.

\begin{thm}[{\cite[Theorem~1.19]{LarsenShalev08}}; {\cite[Theorem~1.5]{LarsenShalev09}}]\label{theorem:larsen.shalev} Let $w \in \free(\N)\setminus \{ 1 \}$. Then there exists $n \in \N_{>0}$ such that \begin{displaymath}
	\forall m \in \N_{\geq n} \colon \qquad w(\Alt(m))^{2} \, = \, \Alt(m) .
\end{displaymath} \end{thm}

\begin{proof}[Proof of Theorem~\ref{theorem:verbal.simplicity}] Let $w \in \free(\N)\setminus \{ 1 \}$. Due to Theorem~\ref{theorem:larsen.shalev}, there exists $n \in \N_{>0}$ such that \begin{equation}\label{eq:larsen.shalev}
	\forall m \in \N_{\geq n} \colon \qquad w(\Alt(m))^{2} \, = \, \Alt(m) .
\end{equation} Define \begin{displaymath}
	m \, \defeq \, \begin{cases}
			\, 4n & \text{if } \cha(R) \ne 2, \\
			\, 8n & \text{if } \cha(R) = 2.
		\end{cases}
\end{displaymath} Using Lemma~\ref{lemma:nondiscrete}, we find a family of matrix units $s \in \M_{m}(R)$ for $R$. One readily checks that $s_{11},\ldots,s_{mm} \in \E(R)$ are pairwise orthogonal and that $\rk_{R}(s_{ii}) = \tfrac{1}{m}$ for each $i \in \{ 1,\ldots,m \}$. Consider the unital ring $S \defeq s_{11}Rs_{11}$. According to Remark~\ref{remark:matrix.units}, \begin{displaymath}
	\phi \colon \, \M_{m}(S) \, \longrightarrow \, R, \quad a \, \longmapsto \, \sum\nolimits_{i,j=1}^{m} s_{i1}a_{ij}s_{1j}
\end{displaymath} is a (necessarily unital) ring isomorphism, thus restricts to a group isomorphism from $\GL_{m}(S) = \GL(\M_{m}(S))$ to $\GL(R)$. Moreover, there is a natural group homomorphism $\iota \colon \Sym(m) \to \GL_{m}(S)$ given by \begin{displaymath}
	\iota(\sigma)_{ij} \, \defeq \, \begin{cases}
			\, 1 & \text{if } \sigma(j) = i, \\
			\, 0 & \text{otherwise}
		\end{cases}
\end{displaymath} for all $\sigma \in \Sym(m)$ and $i,j \in \{ 1,\ldots,m \}$. In turn, \begin{displaymath}
	\psi \colon \, \Sym(m) \, \longrightarrow \, \GL(R), \quad \sigma \, \longmapsto \, \phi(\iota(\sigma))
\end{displaymath} constitutes a well-defined group homomorphism, whence \begin{equation}\label{eq:larsen.shalev.2}
	\psi(\Alt(m)) \, \stackrel{\eqref{eq:larsen.shalev}}{=} \, \psi\!\left(w(\Alt(m))^{2}\right) \, = \, w(\psi(\Alt(m)))^{2} \, \subseteq \, w(\GL(R))^{2} .
\end{equation} Consider the permutation \begin{displaymath}
	\sigma \, \defeq \, \begin{cases}
			\, (1 \ \, 2n+1) \cdots (2n \ \, 4n) & \text{if } \cha(R) \ne 2, \\ 
			\, (1 \ \, 2n+1) \cdots (2n \ \, 4n)(4n+1)\cdots (8n) & \text{if } \cha(R) = 2
		\end{cases}
\end{displaymath} and note that $\sigma \in \I(\Alt(m))$. In particular, \begin{displaymath}
	g \, \defeq \, \psi(\sigma) \, \in \, \I(R) \cap \psi(\Alt(m)) .
\end{displaymath} Furthermore, letting $\textbf{0} \defeq 0_{\M_{2n}(S)}$ and $\textbf{1} \defeq 1_{\M_{2n}(S)}$, we see that \begin{displaymath}
	\iota(\sigma) \, = \, \begin{cases}
			\, \begin{pmatrix} \textbf{0} & \textbf{1} \\ \textbf{1} & \textbf{0} \end{pmatrix} & \text{if } \cha(R) \ne 2, \\ 
			\, \begin{pmatrix} \textbf{0} & \textbf{1} & \textbf{0} & \textbf{0} \\ \textbf{1} & \textbf{0} & \textbf{0} & \textbf{0} \\ \textbf{0} & \textbf{0} & \textbf{1} & \textbf{0} \\ \textbf{0} & \textbf{0} & \textbf{0} & \textbf{1} \end{pmatrix} & \text{if } \cha(R) = 2.
		\end{cases}
\end{displaymath} Accordingly, if $\cha(R) \ne 2$, then \begin{align*}
	\rk_{R}(1-g) \, &= \, \rk_{R}(\phi(1-\iota(\sigma))) \, = \, \rk_{R}\!\left(\! \phi\!\begin{pmatrix} \textbf{1} & -\textbf{1} \\ -\textbf{1} & \textbf{1} \end{pmatrix} \!\right) \\
		& = \, \rk_{R}\!\left( \!\phi \!\left( \!\begin{pmatrix} \textbf{1} & \textbf{0} \\ -\textbf{1} & \textbf{1} \end{pmatrix}\begin{pmatrix} \textbf{1} & \textbf{0} \\ \textbf{0} & \textbf{0} \end{pmatrix}\begin{pmatrix} \textbf{1} & -\textbf{1} \\ \textbf{0} & \textbf{1} \end{pmatrix}\!\right) \!\right) \\
		& = \, \rk_{R}\!\left( \!\phi \!\begin{pmatrix} \textbf{1} & \textbf{0} \\ -\textbf{1} & \textbf{1} \end{pmatrix} \phi\!\begin{pmatrix} \textbf{1} & \textbf{0} \\ \textbf{0} & \textbf{0} \end{pmatrix} \phi \!\begin{pmatrix} \textbf{1} & -\textbf{1} \\ \textbf{0} & \textbf{1} \end{pmatrix}\!\right) \, = \, \rk_{R}\!\left( \! \phi\!\begin{pmatrix} \textbf{1} & \textbf{0} \\ \textbf{0} & \textbf{0} \end{pmatrix} \! \right) \\
		& = \, \rk_{R}(s_{11} + \ldots + s_{2n\,2n}) \, = \, \rk_{R}(s_{11}) + \ldots + \rk_{R}(s_{2n\,2n}) \, = \, \tfrac{2n}{m} \, = \, \tfrac{1}{2};
\end{align*} and if $\cha(R) = 2$, then \begin{align*}
	\rk_{R}(1-g) \, &= \, \rk_{R}(\phi(1-\iota(\sigma))) \, = \, \rk_{R}\!\left(\! \phi\!\begin{pmatrix} \textbf{1} & -\textbf{1} & \textbf{0} & \textbf{0} \\ -\textbf{1} & \textbf{1} & \textbf{0} & \textbf{0} \\ \textbf{0} & \textbf{0} & \textbf{0} & \textbf{0} \\ \textbf{0} & \textbf{0} & \textbf{0} & \textbf{0} \end{pmatrix} \!\right) \\
		& = \, \rk_{R}\!\left( \!\phi \!\left( \!\begin{pmatrix} \textbf{1} & \textbf{0} & \textbf{0} & \textbf{0} \\ -\textbf{1} & \textbf{1} & \textbf{0} & \textbf{0} \\ \textbf{0} & \textbf{0} & \textbf{1} & \textbf{0} \\ \textbf{0} & \textbf{0} & \textbf{0} & \textbf{1} \end{pmatrix}\begin{pmatrix} \textbf{1} & \textbf{0} & \textbf{0} & \textbf{0} \\ \textbf{0} & \textbf{0} & \textbf{0} & \textbf{0} \\ \textbf{0} & \textbf{0} & \textbf{0} & \textbf{0} \\ \textbf{0} & \textbf{0} & \textbf{0} & \textbf{0} \end{pmatrix}\begin{pmatrix} \textbf{1} & -\textbf{1} & \textbf{0} & \textbf{0} \\ \textbf{0} & \textbf{1} & \textbf{0} & \textbf{0} \\ \textbf{0} & \textbf{0} & \textbf{1} & \textbf{0} \\ \textbf{0} & \textbf{0} & \textbf{0} & \textbf{1} \end{pmatrix}\!\right) \!\right) \\
		& = \, \rk_{R}\!\left( \!\phi \!\begin{pmatrix} \textbf{1} & \textbf{0} & \textbf{0} & \textbf{0} \\ -\textbf{1} & \textbf{1} & \textbf{0} & \textbf{0} \\ \textbf{0} & \textbf{0} & \textbf{1} & \textbf{0} \\ \textbf{0} & \textbf{0} & \textbf{0} & \textbf{1} \end{pmatrix} \phi\!\begin{pmatrix} \textbf{1} & \textbf{0} & \textbf{0} & \textbf{0} \\ \textbf{0} & \textbf{0} & \textbf{0} & \textbf{0} \\ \textbf{0} & \textbf{0} & \textbf{0} & \textbf{0} \\ \textbf{0} & \textbf{0} & \textbf{0} & \textbf{0} \end{pmatrix} \phi \!\begin{pmatrix} \textbf{1} & -\textbf{1} & \textbf{0} & \textbf{0} \\ \textbf{0} & \textbf{1} & \textbf{0} & \textbf{0} \\ \textbf{0} & \textbf{0} & \textbf{1} & \textbf{0} \\ \textbf{0} & \textbf{0} & \textbf{0} & \textbf{1} \end{pmatrix}\!\right) \\
		& = \, \rk_{R}\!\left( \! \phi\!\begin{pmatrix} \textbf{1} & \textbf{0} & \textbf{0} & \textbf{0} \\ \textbf{0} & \textbf{0} & \textbf{0} & \textbf{0} \\ \textbf{0} & \textbf{0} & \textbf{0} & \textbf{0} \\ \textbf{0} & \textbf{0} & \textbf{0} & \textbf{0} \end{pmatrix} \! \right) \, = \, \rk_{R}(s_{11} + \ldots + s_{2n\,2n}) \\
		& = \, \rk_{R}(s_{11}) + \ldots + \rk_{R}(s_{2n\,2n}) \, = \, \tfrac{2n}{m} \, = \, \tfrac{1}{4}.
\end{align*} Consequently, \begin{align*}
	\GL(R) \, &\stackrel{\ref{lemma:bounded.normal.generation}}{=} \, \Cl_{\GL(R)}(g)^{32} \, \subseteq \, \Cl_{\GL(R)}(\psi(\Alt(m)))^{32} \, \stackrel{\eqref{eq:larsen.shalev.2}}{\subseteq} \, \Cl_{\GL(R)}\!\left(w(\GL(R))^{2}\right)^{32} \\
		& \subseteq \, \Cl_{\GL(R)}(w(\GL(R)))^{64} \, = \, w(\GL(R))^{64} . \qedhere
\end{align*} \end{proof}


\begin{thebibliography}{50}
	
	\bibitem{bass}
	Hymen Bass, \emph{$K$-theory and stable algebra}. Inst.\ Hautes Études Sci.\ Publ.\ Math.~\textbf{22} (1964), 5--60.
	
	\bibitem{Berberian82}
	Sterling~K.\ Berberian, \emph{The maximal ring of quotients of a finite von Neumann algebra}. Rocky Mountain J.\ Math.~\textbf{12} (1982), no.~1, 149--164.
	
	\bibitem{bergman}
	George~M.\ Bergman, \emph{Generating infinite symmetric groups}. Bull.\ London Math.\ Soc.~\textbf{38} (2006), no.~3, 429--440.
	
	\bibitem{BernardSchneider}
	Josefin Bernard, Friedrich~M.\ Schneider, \emph{Width bounds and Steinhaus property for unit groups of continuous rings}. arXiv: 2412.17480 [math.GR].
	
	\bibitem{CarderiThom}
	Alessandro Carderi, Andreas Thom, \emph{An exotic group as limit of finite special linear groups}. Ann.\ Inst.\ Fourier (Grenoble)~\textbf{68} (2018), no.~1, 257--273.
	
	\bibitem{Cornulier06}
	Yves de Cornulier, \emph{Strongly bounded groups and infinite powers of finite groups}. Comm.\ Algebra~\textbf{34} (2006), no.~7, 2337--2345.
	
	\bibitem{Cornulier07}
	Yves de Cornulier, \emph{Property T for linear groups over rings, after Shalom}. Manuscript, February 2007, 4 pages. \url{https://www.normalesup.org/~cornulier/shalexp.pdf}
	
	
	\bibitem{dickson1}
	Leonard~E.\ Dickson, \emph{The analytic representation of substitutions on a power of a prime number of letters with a discussion of the linear group}. Ann.\ of Math.~\textbf{11} (1896/97), no.~1-6, 161--183.
	
	\bibitem{dickson2}
	Leonard~E.\ Dickson, \emph{Theory of linear groups in an arbitrary field}. Trans.\ Amer.\ Math.\ Soc.~\textbf{2} (1901), no.~4, 363--394.
	
	\bibitem{dowerk}
	Philip~A.\ Dowerk, \emph{Strong uncountable cofinality for unitary groups of von Neumann algebras}. Forum Math.~\textbf{32} (2020), no.~3, 773--781. 
	
	\bibitem{DrosteGoebel}
	Manfred Droste, Rüdiger Göbel, \emph{Uncountable cofinalities of permutation groups}. J.\ London Math.\ Soc.\ (2) \textbf{71} (2005), no.~2, 335--344.
	
	\bibitem{DrosteHolland}
	Manfred Droste, W.~Charles Holland, \emph{Generating automorphism groups of chains}. Forum Math.~\textbf{17} (2005), no.~4, 699--710.
	
	\bibitem{DrosteHollandUlbrich}
	Manfred Droste, W.~Charles Holland, Georg Ulbrich, \emph{On full groups of measure-preserving and ergodic transformations with uncountable cofinalities}. Bull.\ Lond.\ Math.\ Soc.~\textbf{40} (2008), no.~3, 463--472.
	
	\bibitem{Ehrlich56}
	Gertrude Ehrlich, \emph{Characterization of a continuous geometry within the unit group}. Trans.\ Amer.\ Math.\ Soc.~\textbf{83} (1956), 397--416.
		
	\bibitem{elek2013}
	Gábor Elek, \emph{Connes embeddings and von Neumann regular closures of amenable group algebras}. Trans.\ Amer.\ Math.\ Soc.~365 (2013), no.~6, 3019--3039.
	
	\bibitem{elek}
	Gábor Elek, \emph{Lamplighter groups and von Neumann's continuous regular ring}. Proc.\ Amer.\ Math.\ Soc.~144 (2016), no.~7, 2871--2883.
	
	\bibitem{ElekSzabo}
	Gábor Elek, Endre Szabó, \emph{Sofic groups and direct finiteness}. J. Algebra~\textbf{280} (2004), no.~2, 426--434.	
		
	\bibitem{feldman}
	Jacob Feldman, \emph{Isomorphisms of Finite Type II Rings of Operators}. Ann.\ Math.\ (2)~\textbf{63} (1956), no.~3, 565--571.	
		
	\bibitem{fuchs}
	L\'aszl\'o Fuchs, \emph{On a substitution property of modules}. Monatsh.\ Math.~\textbf{75} (1971), 198--204.	
		
	\bibitem{Goodearl78}
	Kenneth~R.\ Goodearl, \emph{Centers of regular self-injective rings}. Pacific J.\ Math.~\textbf{76} (1978), no.~2, 381--395.	
		
	\bibitem{GoodearlBook}
	Kenneth~R.\ Goodearl, \emph{von Neumann regular rings}. Monographs and Studies in Mathematics, 4.~Pitman (Advanced Publishing Program), Boston, Mass.-London, 1979.
	
	
	\bibitem{Halperin62}
	Israel Halperin, \emph{Von Neumann's arithmetics of continuous rings}. Acta Sci.\ Math.\ (Szeged) \textbf{23} (1962), 1--17.  
	
	\bibitem{Halperin68}
	Israel Halperin, \emph{Von Neumann's manuscript on inductive limits of regular rings}. Canad.\ J.\ Math.~\textbf{20} (1968), 477--483.
	
	\bibitem{henriksen}
	Melvin Henriksen, \emph{On a class of regular rings that are elementary divisor rings}. Arch.\ Math.\ (Basel)~\textbf{24} (1973), 133--141.
	
	\bibitem{jordan}
	Camille Jordan, \emph{Trait\'e des substitutions et des \'equations alg\'ebriques}. Gauthier-Villars, Paris, 1870.
	
	\bibitem{KadisonLiu}
	Richard~V.\ Kadison, Zhe Liu, \emph{The Heisenberg Relation – Mathematical Formulations}. SIGMA Symmetry Integrability Geom.\ Methods Appl.~\textbf{10} (2014), Paper 009, 40 pp. 
	
	\bibitem{kassabov}
	Martin Kassabov, \emph{Universal lattices and unbounded rank expanders}. Invent.\ Math.~\textbf{170} (2007), no.~2, 297--326.
	
	\bibitem{KechrisRosendal}
	Alexander~S.\ Kechris, Christian Rosendal, \emph{Turbulence, amalgamation, and generic automorphisms of homogeneous structures}. Proc.\ Lond.\ Math.\ Soc.~(3)~\textbf{94} (2007), no.~2, 302--350.
	
	\bibitem{LarsenShalev08}
	Michael Larsen, Aner Shalev, \emph{Characters of symmetric groups: sharp bounds and applications}. Invent.\ Math.~\textbf{174} (2008), no.~3, 645--687.
	
	\bibitem{LarsenShalev09}
	Michael Larsen, Aner Shalev, \emph{Word maps and Waring type problems}. J.\ Amer.\ Math.\ Soc.~\textbf{22} (2009), no.~2, 437--466. 
	
	\bibitem{LiebeckShalev}
	Martin~W.\ Liebeck, Aner Shalev, \emph{Diameters of finite simple groups: sharp bounds and applications}. Ann.\ of Math.~(2)~\textbf{154} (2001), no.~2, 383--406.
	
	\bibitem{linnell}
	Peter~A.\ Linnell, \emph{Embedding group algebras into finite von Neumann regular rings}. Modules and comodules, 295--300, Trends Math., Birkhäuser Verlag, Basel, 2008.	
	
	\bibitem{LinnellSchick}
	Peter~A.\ Linnell, Thomas Schick, \emph{The Atiyah conjecture and Artinian rings}. Pure Appl.\ Math.\ Q.~\textbf{8} (2012), no.~2, 313--327.
	
	\bibitem{Maeda50}
	Fumitomo Maeda, \emph{Embedding theorem of continuous regular rings}. J.\ Sci.\ Hiroshima Univ.\ Ser.\ A \textbf{14} (1949), 1--7.
	
	\bibitem{MaedaBook}
	Fumitomo Maeda, \emph{Kontinuierliche Geometrien}. Die Grundlehren der mathematischen Wissenschaften in Einzeldarstellungen mit besonderer Berücksichtigung der Anwendungsgebiete, Band 95. Springer-Verlag, Berlin-Göttingen-Heidelberg, 1958.
	
	\bibitem{moore}
	E.~Hastings Moore, \emph{A doubly-infinite system of simple groups}. Bull.\ Amer.\ Math.\ Soc.~\textbf{3} (1893), no.~3, 73--78.
	
	\bibitem{MurrayVonNeumann}
	Francis~J.\ Murray, John von Neumann, \emph{On rings of operators}. Ann.\ of Math.\ (2)~\textbf{37} (1936), no.~1, 116--229.
	
	\bibitem{NeumannExamples}
	John von Neumann, \emph{Examples of continuous geometries}. Proc.\ Nat.\ Acad.\ Sci.\ U.S.A.~\textbf{22} (1936), 101--108.
	
	\bibitem{VonNeumannBook}
	John von Neumann, \emph{Continuous geometry}. Foreword by Israel Halperin. Princeton Mathematical Series, No.~25, Princeton University Press, Princeton, N.J., 1960.
	
	\bibitem{RicardRosendal}
	\'Eric Ricard, Christian Rosendal, \emph{On the algebraic structure of the unitary group}. Collect.~Math.~\textbf{58} (2007), no.~2, 181--192.
	
	\bibitem{rosendal}
	Christian Rosendal, \emph{A topological version of the Bergman property}. Forum Math.~\textbf{21} (2009), no.~2, 299--332.
	
	\bibitem{SchneiderGAFA}
	Friedrich~M.\ Schneider, \emph{Concentration of invariant means and dynamics of chain stabilizers in continuous geometries}. Geom.\ Funct.\ Anal.~\textbf{33} (2023), no.~6, 1608--1681. 
	
	\bibitem{SchneiderIMRN}
	Friedrich~M.\ Schneider, \emph{Group von Neumann algebras, inner amenability, and unit groups of continuous rings}. Int.\ Math.\ Res.\ Not.\ IMRN (2024), no.~8, 6422--6446.
	
	\bibitem{Schneider25}
	Friedrich~M.\ Schneider, \emph{Geometric properties of unit groups of von Neumann's continuous rings}. J.\ Algebra~\textbf{712} (2027), 654--694. 
	
	\bibitem{jschneider}
	Jakob Schneider, \emph{On groups with unbounded Cayley graphs}. J.\ Algebra~\textbf{564} (2020), 317--324.
	
	\bibitem{segal}
	Dan Segal, \emph{Words: notes on verbal width in groups}. London Math.\ Soc.\ Lecture Note Ser., 361. Cambridge University Press, Cambridge, 2009. 
	
	\bibitem{shalom}
	Yehuda Shalom, \emph{Bounded generation and Kazhdan's property (T)}. Inst.\ Hautes Études Sci.\ Publ.\ Math.~\textbf{90} (1999), 145--168.
	
	\bibitem{ThomasZlapetal}
	Simon Thomas, Jindřich Zapletal, \emph{On the Steinhaus and Bergman properties for infinite products of finite groups}. Confluentes Math.~\textbf{4} (2012), no.~2, 1250002, 26 pp.
	
	\bibitem{TolstykhLinear}
	Vladimir~A.\ Tolstykh, \emph{Infinite-dimensional general linear groups are groups of finite width}. Siberian Math.\ J.~\textbf{47} (2006), no.~5, 950--954.
	
	\bibitem{TolstykhFree}
	Vladimir~A.\ Tolstykh, \emph{On the Bergman property for the automorphism groups of relatively free groups}. J.\ London Math.\ Soc.~(2) \textbf{73} (2006), no.~3, 669--680.
	
	\bibitem{utumi}
	Yuzo Utumi, \emph{On continuous rings and self injective rings}. Trans.\ Amer.\ Math.\ Soc.~\textbf{118} (1965), 158--173.
	
	\bibitem{vaserstein}
	Leonid~N.\ Vaserstein, \emph{The stable range of rings and the dimension of topological spaces}. Functional Anal.\ Appl.~\textbf{5} (1971), 102--110.
	
	\bibitem{vaserstein2}
	Leonid~N.\ Vaserstein, \emph{Bass's first stable range condition}. Proceedings of the Luminy conference on algebraic $K$-theory (Luminy, 1983). J.\ Pure Appl.\ Algebra~\textbf{34} (1984), no.~2-3, 319--330.
	
\end{thebibliography}
\end{document}